\documentclass{article}
\usepackage[utf8]{inputenc}

\usepackage{authblk}

\usepackage{comment,fullpage,amssymb,stmaryrd,amsmath,amsthm,enumerate,hyperref,mathrsfs,tikz-cd,appendix}
\usepackage[style=alphabetic, maxnames=10]{biblatex}
\DeclareMathOperator{\gr}{gr}
\DeclareMathOperator{\End}{End}

\DeclareMathOperator{\Hom}{Hom}
\DeclareMathOperator{\rad}{rad}
\DeclareMathOperator{\id}{id}
\DeclareMathOperator{\im}{im}

\DeclareMathOperator{\length}{length}

\DeclareMathOperator{\Dir}{Dir}

\DeclareMathOperator{\GL}{GL}
\DeclareMathOperator{\diag}{diag}
\DeclareMathOperator{\M}{M}
\DeclareMathOperator{\Prolif}{Prolif}
\DeclareMathOperator{\Spec}{Spec}

\theoremstyle{plain}
\newtheorem{theorem}{Theorem}[section]
\newtheorem{theorem-definition}{Theorem-Definition}[section]

\newtheorem{corollary}[theorem]{Corollary}

\newtheorem{lemma}[theorem]{Lemma}
\newtheorem{proposition}[theorem]{Proposition}

\theoremstyle{definition}

\newtheorem{definition}[theorem]{Definition}
\newtheorem{proposition-definition}[theorem]{Proposition-Definition}

\newtheorem{example}[theorem]{Example}

\theoremstyle{remark}

\newtheorem{remark}[theorem]{Remark}

\tikzset{%
  symbol/.style={
    draw=none,
    every to/.append style={
      edge node={node [sloped, allow upside down, auto=false]{$#1$}}
    },
  },
}

\makeatletter
\newcommand\email[2][]%
   {\newaffiltrue\let\AB@blk@and\AB@pand
      \if\relax#1\relax\def\AB@note{\AB@thenote}\else\def\AB@note{\relax}%
        \setcounter{Maxaffil}{0}\fi
      \begingroup
        \let\protect\@unexpandable@protect
        \def\thanks{\protect\thanks}\def\footnote{\protect\footnote}%
        \@temptokena=\expandafter{\AB@authors}%
        {\def\\{\protect\\\protect\Affilfont}\xdef\AB@temp{#2}}%
         \xdef\AB@authors{\the\@temptokena\AB@las\AB@au@str
         \protect\\[\affilsep]\protect\Affilfont\AB@temp}%
         \gdef\AB@las{}\gdef\AB@au@str{}%
        {\def\\{, \ignorespaces}\xdef\AB@temp{#2}}%
        \@temptokena=\expandafter{\AB@affillist}%
        \xdef\AB@affillist{\the\@temptokena \AB@affilsep
          \AB@affilnote{}\protect\Affilfont\AB@temp}%
      \endgroup
       \let\AB@affilsep\AB@affilsepx
}
\makeatother

\title{Bushnell-Reiner zeta functions over two-dimensional semilocal rings}
\author{Sean B. Lynch}

\begin{document}

\maketitle

\begin{abstract}
Lustig gave an infinite product formula for the zeta function of a commutative two-dimensional regular local ring with finite residue field. We extend this to the noncommutative setting with a method based on filtration by an invertible ideal. One application gives an abstract two-dimensional analogue of Hey's formula. Another application provides effective formulae for zeta functions over Rump's two-dimensional regular semiperfect rings. In the appendices, we supplement these two-dimensional applications with requisite one-dimensional calculations.
\end{abstract}

\section{Introduction}

Orders over discrete valuation rings are classical objects in algebraic number theory \cite{Rei03} with a rich theory of zeta functions \cite{br_survey}. Recall that discrete valuation rings are nothing but one-dimensional regular local rings. Over two-dimensional regular local rings, orders are not as well understood. See \cite{Ramras1969,Artin1986,reiten&bergh,chan&ingalls_invent,chan&ingalls_ant} for some key developments. Zeta functions over such two-dimensional orders were first studied in the author's thesis \cite{lynch_thesis}. In this work, we streamline the method and yield more general local applications. In another direction, Daniel Chan and the author \cite{chan2024zetafunctionsorderssurfaces} investigated global applications of a special local formula (Corollary \ref{cor_hom_slice}).

Let $R$ be a ring and $M$ be a left $R$-module. If, for each positive integer $n$, there is only a finite number $a_{M,n}$ of submodules of $M$ with index $n$, then we may form the Dirichlet series $\zeta(M;s)=\sum_{n=1}^\infty a_{M,n} n^{-s}$. This is called the \textit{Solomon zeta function of $M$}. Motivated by integral representation theory, Solomon studied the case where $M$ is a lattice over a $\mathbb{Z}_p$-order $R$ \cite{Sol_77_first,Sol_79_second}. After proving Solomon's first conjecture \cite{Bushnell1980}, Bushnell and Reiner studied related $L$-functions and zeta functions \cite{bushnell&reiner_crelle_1,bushnell&reiner_crelle_2,bushnell82,bushnell84,bushnell86}, finally turning to the \textit{Bushnell-Reiner zeta function} $Z(M;z)$ \cite{Bushnell1987}. This zeta function is more refined and categorical than its Solomonian counterpart. A similarly refined and categorical zeta function appears in Iyama's proof of Solomon's second conjecture \cite{Iya03}. In Section \ref{section_foundations}, we construct Bushnell-Reiner zeta functions over \textit{left arithmetical rings} on Knopfmacher's foundation of abstract Dirichlet algebras associated to arithmetical semigroups \cite{knopfmacher_abstract}. We use the Grothendieck group of the category of finite length left modules over left arithmetical rings to construct the requisite arithmetical semigroups and recover Solomon zeta functions by functoriality.

Solomon's method relies heavily on \textit{partial zeta functions} $\zeta(M,N;s)$, which enumerate only those finite index submodules of $M$ that are isomorphic to a given module $N$. We similarly define their Bushnell-Reiner counterparts $Z(M,N;z)$ and use them extensively in our own method. 

We take an invertible ideal $I$ in the Jacobson radical of a left arithmetical ring $R$, and we take any finitely generated projective left $R$-module $M$. Then we consider the $I$-adic filtration on each finite length quotient $M/X$, and we study $Z(M;z)$ via the associated graded modules $\gr_I(M/X)$. Under the assumption that the finite colength submodules of $M/IM$ are all projective over $R/I$, we present our \textit{proliferation formula} in Section \ref{section_prolif}. This expresses $Z(M;z)$ as a (possibly infinite) sum of infinite products of $Z(P_{j+1},P_j;-)$ for finite colength submodules $P_j\subseteq P_{j+1}\subseteq M/IM$. In particular, when the finite colength submodules of $M/IM$ are all isomorphic, the proliferation formula expresses $Z(M;z)$ as an infinite product of $Z(M/IM;-)$. These results are stated in Theorem \ref{thm_prolif} and Corollary \ref{corollary_single_sliver}.

In later sections, we explore applications of the proliferation formula to modules over two-dimensional semilocal orders. Before moving on to these, we collect some tangential applications at the end of Section \ref{section_prolif}. In Example \ref{example_voll}, we give a decomposition formula for one-dimensional zeta functions that has proved useful, at least in some special cases (cf. \cite[E.g.~2.20]{Voll}); our work in Section \ref{section_rump} may be seen as a two-dimensional analogue. In Example \ref{example_rossmann}, we recover Rossmann's formula for the Solomon zeta function enumerating finite index ideals of the formal power series ring $\mathbb{Z}\llbracket t\rrbracket$.

In Appendix \ref{section_hey}, we prove an abstract version of Hey's formula \cite{Hey27} by adapting a method of Solomon \cite[\S.~2--3]{Sol_77_first}. Then, in Section \ref{section_lift_hey}, we combine our abstract Hey formula with our proliferation formula to obtain Theorem \ref{thm_lifted_hey}. This is an explicit infinite product formula for a Bushnell-Reiner zeta function in terms of Artin-Wedderburn data. It is surprising how explicit this formula is, given how abstract our assumptions are. Under additional assumptions, we obtain a pleasant formula for the Solomon zeta function in Corollary \ref{cor_hom_slice}. As previously stated, Daniel Chan and the author explore global applications of this specialised formula in \cite{chan2024zetafunctionsorderssurfaces}. Moreover, in Example \ref{example_lustig}, we see that Corollary \ref{cor_hom_slice} implies Lustig's classical formula for the Solomon zeta function enumerating finite index ideals of a commutative two-dimensional regular local ring with finite residue field.

We consider Rump's (left) \textit{two-dimensional regular semiperfect rings} \cite{rump_reg} in Section \ref{section_rump}. On the one hand, their definition plays nicely with the conditions of our proliferation formula. On the other hand, Rump classified such rings in terms of rank two \textit{almost regular valuations} \cite[Theorem~4.1]{rump_reg} and showed that they are orders of global dimension two in (Artinian) semisimple rings \cite[Proposition~1.4 and Theorem~4.2]{rump_reg}. In a two-dimensional regular semiperfect ring $R$, we may choose an invertible ideal $I$ in the Jacobson radical such that $R/I$ is a one-dimensional regular semiperfect ring, so we have proliferation formulae involving partial zeta functions over $R/I$. Michler classified the one-dimensional regular semiperfect rings as certain hereditary orders in semisimple rings \cite{Michler1969} (see \cite[Proposition~1.6]{rump_reg}), and they look just like the classical hereditary orders \cite[\S.~39]{Rei03} that they generalise.  

In Appendix \ref{appendix_brz_hereditary}, we compute the requisite partial Bushnell-Reiner zeta functions over one-dimensional regular semiperfect rings. Bushnell and Reiner computed partial Solomon zeta functions over classical hereditary orders in quaternion algebras and remarked that their calculations became cumbersome over classical hereditary orders in higher-dimensional algebras \cite{bushnell&reiner_texas}. Denert found explicit combinatorial formulae over all such orders \cite{denert,denert_second}, a special case of which inspired work on \textit{Denert's permutation statistic} \cite{foata&zeilberger,carnevale_denert}. Our formulae are, at least \textit{a priori}, more general and refined than Denert's. Moreover, we express partial Bushnell-Reiner zeta functions over hereditary orders in terms of polynomials associated to filtered finite-dimensional vector spaces over finite fields. These filtered vector spaces occur naturally in the construction of hereditary orders as chain orders, a construction that is foundational in Bushnell's work on representations of $\GL_n$ over $p$-adic fields \cite[\S.~1]{bushnell_super} (see also the earlier joint work \cite[\S.~1.2]{bushnell&frohlich_paper}). The supercuspidal representations are induced from maximal open compact-mod-centre subgroups, which are the normalisers of principal orders. Moreover, these orders are endomorphism rings for periodic chains of $p$-adic lattices that lift filtered finite vector spaces with evenly spaced dimensions.

Before briefly mentioning some analytic results that are tangential to our paper, we want to explicitly mention the new ideas, insights, and methods used in this paper. Our foundations for Bushnell-Reiner and Solomon zeta functions are new (Section \ref{section_foundations}). In particular, our notion of a left arithmetical ring is new, as is our application of Knopfmacher's abstract analytic number theory to Bushnell-Reiner and Solomon zeta functions. Although our filtration method (Section \ref{section_prolif}) has its origins in Lustig's work \cite{Lus56}, we inject many new ideas to make it work in our general setting. For example, we start with Bushnell-Reiner zeta functions, filter by an invertible ideal, and pass to partial Bushnell-Reiner zeta functions. Lustig's case (Example \ref{example_lustig}) is very special and offers many shortcuts, of which we could not take advantage. As a consequence of our general method, we derive new formulae for zeta functions over two-dimensional orders (Section \ref{section_lift_hey} and \ref{section_rump}). Finally, we give a new method to compute Bushnell-Reiner zeta functions over hereditary orders (Appendix \ref{appendix_brz_hereditary}). Previously, we only had formulae for their Solomon zeta functions.

 Let $R$ be a commutative Noetherian semilocal ring whose residue fields are all finite. Segal showed that if $M$ is a finitely generated $R$-module, then $\zeta(M;s)$ converges at some complex number $s$ if and only if the Krull dimension $\dim(M)\le 2$ \cite[Thm.~1]{Segal1997}. Earlier results of Berndt \cite{Ber69} and Witt (published at the end of \cite{Ber69}) give this dichotomy in the special case when $M={}_R R$ and $R$ is local (see \cite{Ber03} for an English summary). Segal's result extends to the following noncommutative setting. If $\Lambda$ is an $R$-algebra such that the modules ${}_R \Lambda$ and ${}_\Lambda M$ are both finitely generated, then $\zeta({}_\Lambda M;s)$ converges at some complex number $s$ if and only if $\dim({}_R M)\le 2$. Moreover, it follows from the Bushnell-Reiner-Solomon theorem \cite{Bushnell1980} that the Solomon zeta functions of orders in dimension one extend meromorphically to the entire complex plane. In dimension two, however, these zeta functions may have natural boundary! Such analytic results are covered in the author's PhD thesis; they suggest that the two-dimensional case is special.
 
\textbf{Acknowledgments}. We thank Daniel Chan and Ivan Fesenko for many helpful conversations.

\section{Bushnell-Reiner zeta functions}\label{section_foundations}

We begin with some definitions from \cite[Chapter 1--2]{knopfmacher_abstract}.

\begin{definition}
    A \textit{countably normed set} is a pair $(\mathcal{S},\mathscr{N})$ consisting of a countable set $\mathcal{S}$ and a function $\mathscr{N}:\mathcal{S}\to\mathbb{R}_{>1}$ such that, for each $x\in\mathbb{R}_{>1}$, there are at most finitely many $i\in\mathcal{S}$ with $\mathscr{N}(i)\le x$. Given a countably normed set $(\mathcal{S},\mathscr{N})$, the \textit{arithmetical semigroup generated by $(\mathcal{S},\mathscr{N})$} is the pair $(G,\mathscr{N})$, where $G$ is the free commutative monoid on $\mathcal{S}$ and $\mathscr{N}:G\to\mathbb{R}_{\ge 1}$ is the extension of $\mathscr{N}:\mathcal{S}\to\mathbb{R}_{>1}$ to a multiplicative monoid homomorphism. We write $G$ multiplicatively. Moreover, by uniqueness, there should be no confusion in us also writing $\mathscr{N}$ for the multiplicative extension to $G$. Furthermore, an \textit{arithmetical semigroup} is an arithmetical semigroup generated by some countably normed set. 
\end{definition}

\begin{definition}
    Let $(G,\mathscr{N})$ be an arithmetical semigroup. The \textit{Dirichlet algebra of $(G,\mathscr{N})$} is the topological ring 
    $$\Dir(G,\mathscr{N})=\{\sum_{g\in G} a_g \, g: \mbox{ each } a_g\in\mathbb{C}\}$$
    with the following addition, multiplication and nonarchimedean absolute value.
    \begin{itemize}
        \item Addition:
    $$(\sum_{g\in G} a_g \, g)+(\sum_{g\in G} b_g \, g)=\sum_{g\in G} (a_g+b_g) \, g$$
    \item Multiplication:
    $$(\sum_{g\in G} a_g \, g)(\sum_{g\in G} b_g \, g)=\sum_{g\in G} (\sum_{\substack{h,i\in G\\ g=hi}}a_{h} b_{i}) \, g$$
    \item Nonarchimedean absolute value: 
    $$\lvert\sum_{g\in G} a_g \, g\rvert_{\mathscr{N}}=\begin{cases}\frac1{\min\{\mathscr{N}(g):a_g\ne 0\}} & \mbox{ if } \sum_{g\in G} a_g \, g\ne 0\\ 0 & \mbox{ otherwise}\end{cases}$$
    \end{itemize}
\end{definition}

Now, let $R$ be a \textit{left arithmetical ring}. That is, $R$ has each of the following properties:
\begin{enumerate}[(a)]
    \item every simple left $R$-module has finitely many elements;
    \item every simple left $R$-module is finitely presented;
    \item for every positive integer $n$, there are only finitely many (up to isomorphism) simple left $R$-modules with $n$ elements.
\end{enumerate}

The Grothendieck group $\tilde{G}$ of finite length left $R$-modules is free abelian on the set $\mathcal{S}$ of isomorphism classes of simple left $R$-modules. Let $G\subset \tilde{G}$ be the free commutative monoid on $\mathcal{S}$. We write $\tilde{G}$, and thus also $G$, multiplicatively.

Let $M$ be a finitely generated left $R$-module. For any $i\in\mathcal{S}$, we claim that there are only finitely many submodules $X$ of $M$ with $M/X\in i$. This is because the number of such $X$ is bounded above by $|\Hom(M,S)|$, where $S\in i$ is any representative, which is finite by (a). Then, by (b), we see that any maximal submodule of $M$ is finitely generated. Thus, it follows by induction on the number of composition factors that, for any $g\in G$, there are at most finitely many finite colength submodules $X$ of $M$ with $g=[M/X]\in G$.

Define $\mathscr{N}:\mathcal{S}\to\mathbb{R}_{>1}$ by $\mathscr{N}(i)=|S|$ for any $S\in i\in\mathcal{S}$. Then $(\mathcal{S},\mathscr{N})$ is a countably normed set by (c). We have the arithmetical semigroup $(G,\mathscr{N})$ generated by $(\mathcal{S},\mathscr{N})$, and we have the associated Dirichlet algebra $\Dir=\Dir(G,\mathscr{N})$. The previous paragraph implies that we have a convergent series (with respect to the nonarchimedean absolute value)
$$Z(M)=\sum_X [M/X]\in\Dir$$
where the sum is over all finite colength submodules $X$ of $M$. We call this the \textit{Bushnell-Reiner zeta function of $M$} \cite{Bushnell1987}. When $\mathcal{S}$ is finite, $\Dir$ is isomorphic to the ring of formal power series $\mathbb{C}\llbracket z_1,\ldots,z_{|\mathcal{S}|}\rrbracket$ with the $(z_1,\ldots,z_{|\mathcal{S}|})$-adic topology. Bushnell and Reiner defined their zeta functions in this finite setting.

To discuss functoriality, we now write $\mathcal{S}=\mathcal{S}(R)$, $G=G(R)$, $\mathscr{N}=\mathscr{N}_{R}$, and $\Dir=\Dir(R)$. Let $A$ be another left arithmetical ring for which we have a ring homomorphism $R\to A$. So, we have two arithmetical semigroups $(G(R),\mathscr{N}_{R})$ and $(G(A),\mathscr{N}_{A})$. For each $i\in\mathcal{S}(A)$, let $\mathscr{N}_{A/R}(i)=[S]\in G(R)$ for any $S\in i$. Then $\mathscr{N}_{A/R}$ extends uniquely to a monoid homomorphism $G(A)\to G(R)$ satisfying $\mathscr{N}_{A}=\mathscr{N}_{R}\circ\mathscr{N}_{A/R}$. In this sense, $\mathscr{N}_{A/R}:(G(A),\mathscr{N}_{A})\to (G(R),\mathscr{N}_{R})$ is a \textit{morphism of arithmetical semigroups}. This, in turn, induces a topological ring homomorphism $\Dir(\mathscr{N}_{A/R}):\Dir(A)\to\Dir(R)$. Then we see that $R\mapsto \Dir(R)$ gives a contravariant functor from the category of left arithmetical rings, with ring homomorphisms, to the category of topological rings.

Note that $\Dir(\mathbb{Z})$ is just the ordinary formal Dirichlet algebra. So, in particular, we have a topological ring homomorphism $\Dir(\mathscr{N}_{R/\mathbb{Z}}):\Dir(R)\to\Dir(\mathbb{Z})$ that takes the Bushnell-Reiner zeta function of $M$ to the \textit{Solomon zeta function of $M$} \cite{Sol_77_first}.

For another finitely generated left $R$-module $N$, we have the \textit{partial Bushnell-Reiner zeta function} 
$$Z(M,N)=\sum_{X\simeq N} [M/X]\in\Dir(R)$$
where the sum is over all finite colength submodules $X$ of $M$ such that $X\simeq N$. We may similarly recover the \textit{partial Solomon zeta function}. Note that the (partial) Bushnell-Reiner zeta function of $M$ (and $N$) depends only on the isomorphism class of $M$ (and of $N$).

The class of left arithmetical rings is closed under finite direct products and factors. Bushnell-Reiner and Solomon zeta functions, partial or not, are all compatible with such decompositions in the obvious way (cf. \cite[Lemma~1]{Sol_77_first} and Remark \ref{remark_finite_products}). The class of left arithmetical rings is also closed under Morita equivalence. 

\begin{proposition}
    Let $A$ be a ring that is Morita equivalent to a left arithmetical ring $R$. Then $A$ is also left arithmetical. 
\end{proposition}
\begin{proof}
    Under this equivalence, every simple left $A$-module $T$ corresponds to a simple left $R$-module $S$, and ${}_{A}A$ corresponds to a progenerator ${}_R P$. We have $|T|=|\Hom_{A}(A,T)|=|\Hom_R(P,S)|$, so $A$ satisfies property (a) because $R$ does. We also readily see that $A$ satisfies property (b) because $R$ does. Now, let's take a left $R$-module epimorphism $P^{\oplus m}\twoheadrightarrow R$ for some positive integer $m$. Then $|S|^{\frac1{m}}=|\Hom_R(R,S)|^{\frac1{m}}\le|\Hom_{R}(P,S)|=|T|$. Since the correspondence $T\mapsto S$ is bijective on isomorphism classes, $A$ satisfies property (c) because $R$ does.
\end{proof}

Bushnell-Reiner zeta functions, partial or not, are compatible with Morita equivalence in an obvious way. Solomon zeta functions are compatible with Morita equivalence via a free progenerator \cite[Theorem~2.1]{Wittmann2004}. A general Morita equivalence may send simple modules of the same cardinality to simple modules of different cardinalities (Example \ref{example_morita_equivalence}), in which case there is not an obvious relationship between Solomon zeta functions before and after applying the Morita equivalence.

Recall that a ring $R$ with Jacobson radical $J$ is \textit{semilocal} if $R/J$ is Artinian (semisimple).

\begin{proposition}\label{prop_slzr}
    Let $R$ be a ring with Jacobson radical $J$. Then $R$ is left arithmetical and semilocal if and only if $R/J$ is finite and ${}_R J$ is finitely generated.
\end{proposition}
\begin{proof}
    Suppose that $R/J$ is finite and ${}_R J$ is finitely generated. Then $R$ is semilocal, and thus also satisfies property (c). Since every simple left $R$-module is isomorphic to a direct summand of ${}_{R}R/J$, we see that $R$ satisfies properties (a) and (b). 

    Conversely, suppose that $R$ is left arithmetical and semilocal. Then ${}_R R/J$ is a finite direct sum of simples, and is thus finite by (a). By (b), ${}_R R/J$ is finitely presented, and hence ${}_R J$ is finitely generated. 
\end{proof}

\begin{proposition}\label{arithmetical_algebra}
    Let $R$ be a commutative arithmetical ring, and let $A$ be an $R$-algebra that is finitely generated as an $R$-module. Then $A$ is arithmetical.  
\end{proposition}
\begin{proof}
    Let $S$ be a simple left $A$-module. Then $D=\End_A(S)$ is a division $R$-algebra. Let $F$ be the image of $R\to D$. Then $F$ is a field. Therefore, we have $F\simeq R/\mathfrak{m}$ for some maximal ideal $\mathfrak{m}$ of $R$, and thus $F$ is finite. Since $S$ is finitely generated over $R$, it follows that $S$ is finite dimensional over $F$. Hence, $S$ is finite. Let $\mathfrak{M}$ be a maximal left ideal of $A$ with $S\simeq A/\mathfrak{M}$. Since $A$ is finitely generated over $R$, a left arithmetical ring, $\mathfrak{M}$ is finitely generated over $R$ and thus also over $A$. For any positive integer $n$, there are at most finitely many $R$-submodules of $A$ with index $n$. In particular, there are at most finitely many maximal left ideals of $A$ with index $n$. The right-sided properties similarly hold.
\end{proof}

\begin{proposition}\label{prop_quotient}
    Let $I$ be an ideal of a ring $R$. If $R$ is left arithmetical, then so is the quotient $R/I$. The converse holds when $I$ is contained in the Jacobson radical $J$ of $R$ and $I$ is finitely generated as a left $R$-module.
\end{proposition}
\begin{proof}
    $\mathfrak{m}\mapsto\mathfrak{m}/I$ gives an index-preserving bijection from the set of maximal left ideals $\mathfrak{m}$ of $R$ containing $I$ to the set of maximal left ideals of $R/I$. Moreover, if $\mathfrak{m}$ is finitely generated over $R$, then $\mathfrak{m}/I$ is finitely generated over $R/I$. These observations give our first claim. If $I$ is contained in the Jacobson radical of $R$, then every maximal left ideal of $R$ contains $I$. Furthermore, consider a short exact sequence of left $R$-modules $0\to I\to\mathfrak{m}\to\mathfrak{m}/I\to 0$. If $I$ is finitely generated as a left $R$-module and $\mathfrak{m}/I$ is finitely generated as a left $R/I$-module, then $\mathfrak{m}$ is finitely generated as a left $R$-module. These additional observations give our second claim. 
\end{proof}

\begin{example}
    Let $R$ be a commutative Noetherian local ring with finite residue field. By Proposition \ref{prop_slzr}, $R$ is an arithmetical ring. Therefore, by Proposition \ref{arithmetical_algebra}, any $R$-algebra that is finitely generated as an $R$-module is arithmetical.
\end{example}

\begin{example}
    Let $R$ be a commutative finitely generated $\mathbb{Z}$-algebra. It is well known that $R$ is a Noetherian ring satisfying conditions (a) and (c). Therefore, by Proposition \ref{arithmetical_algebra}, any $R$-algebra that is finitely generated as an $R$-module is arithmetical.
\end{example}

\begin{example}
    Let $R$ be a left arithmetical ring, let $\sigma$ be a surjective ring endomorphism of $R$, and let $A=R\llbracket x;\sigma\rrbracket$ be the skew power series ring \cite[Example~1.7, p.~9]{Lam_first}. Then $I=Ax=xA$ is an ideal of $A$ contained in the Jacobson radical of $A$. By Proposition \ref{prop_quotient}, $A$ is left arithmetical.
\end{example}

\section{Proliferation formula}\label{section_prolif}

Let $R$ be a left arithmetical ring and $M$ be a finitely generated projective left $R$-module. 

Recall that an ideal $I$ of $R$ is \textit{invertible} if $I\otimes_R-$ gives an auto-equivalence on the category of left $R$-modules (this is actually left-right symmetric). Then we have an overring $\tilde{R}\supseteq R$ obtained by inverting all of the invertible ideals of $R$ \cite{rump_cm} (see also \cite{Ste75}). Similarly, we have a left $\tilde{R}$-module $\tilde{M}\supseteq M$. As such, for any $R$-submodule $X\subseteq M$, any nonnegative integer $j$ and any invertible ideal $I$ of $R$, we have an $R$-submodule $I^{-j}X\subseteq\tilde{M}$. We need this to define the operator $\mathscr{R}$ below.

\begin{itemize}
\item Suppose that there exists an invertible ideal $I$ of $R$, contained in the Jacobson radical $J$ of $R$, such that the finite colength submodules of $M/IM$ are all projective over $R/I$.
\end{itemize}

Since $I\subseteq J$, we naturally identify the isomorphism classes $\mathcal{S}(R/I)=\mathcal{S}(R)$ of simple left modules as well as the norms $\mathscr{N}_{R/I}=\mathscr{N}_R$ on them. Accordingly, the associated free commutative monoids $G(R/I)=G(R)$ and Dirichlet algebras $\Dir(R/I)=\Dir(R)$ are also naturally identified. For simplicity, we write these objects as $\mathcal{S},\mathscr{N},G,\Dir$. Since $I$ is invertible, $I\otimes_R-$ induces an automorphism of the Grothendieck group $\tilde{G}$ of the category of finite length left $R$-modules, which restricts to a monoid automorphism of $G$, and further restricts to a permutation on $\mathcal{S}$.

\begin{proposition}\label{prop_graded}
    Let $X$ be a submodule of $M$. For each $j$, let
    $$Y_j=\frac{(M\cap I^{-j}X)+IM}{IM}.$$
    Then the $j$-th piece of the $I$-adic associated graded module $\gr_I(M/X)$ is isomorphic to 
    $$I^j\otimes_R \frac{M/IM}{Y_j}.$$
\end{proposition}
\begin{proof}
We compute that
\begin{align*}
    \frac{I^j(M/X)}{I^{j+1}(M/X)}&=\frac{(X+I^jM)/X}{(X+I^{j+1}M)/X}\cong\frac{X+I^jM}{X+I^{j+1}M}=\frac{(X+I^{j+1}M)+I^jM}{X+I^{j+1}M}\cong\frac{I^jM}{I^jM\cap(X+I^{j+1}M)}\\&=\frac{I^jM}{I^j((M\cap I^{-j}X)+IM)}\cong I^j\otimes_R \frac{M}{(M\cap I^{-j}X)+IM}\\&\cong I^j\otimes_R \frac{M/IM}{Y_j}.
\end{align*}
The two steps on the second line use the invertibility of $I$ \cite[\S.~1]{rump_cm}.
\end{proof}

Then $X\mapsto Y_0\subseteq Y_1\subseteq\cdots$ defines a function $\mathscr{F}$ from the set $\mathcal{X}$ of finite colength submodules of $M$ to the set $\mathcal{Y}$ of ascending chains of finite colength submodules of $M/IM$ that eventually stabilise to $M/IM$. On $\mathcal{X}$, we have the operator $\mathscr{R}:X\mapsto M\cap I^{-1}X$. On $\mathcal{Y}$, we have the shift operator $\mathscr{S}:Y_0\subseteq Y_1\subseteq\cdots\mapsto Y_1\subseteq Y_2\subseteq\cdots$. They satisfy $\mathscr{F}\mathscr{R}=\mathscr{S}\mathscr{F}$. We shall also write $\bar{Y}_j=Y_{j+1}/Y_j$ for each $j$.

\begin{lemma}\label{lemma1}
   Let $Y\in\mathcal{Y}$. Then $[\mathscr{R}X/X]=\prod_{k=0}^{\infty}[I^k\otimes \bar{Y}_k]\in G$ for every $X\in \mathscr{F}^{-1}Y$.
\end{lemma}
\begin{proof}
    By Proposition \ref{prop_graded}, $$[M/X]=\prod_{j=0}^\infty \left[I^j\otimes \frac{M/IM}{Y_j}\right]=\prod_{j=0}^\infty \prod_{k=j}^\infty \left[I^j\otimes \bar{Y}_k\right]=\prod_{k=0}^\infty\prod_{j=0}^k[I^j\otimes \bar{Y}_k].$$
    Similarly, 
    $$[M/\mathscr{R}X]=\prod_{k=0}^{\infty}\prod_{j=0}^k[I^j\otimes \bar{Y}_{k+1}]=\prod_{k=1}^\infty\prod_{j=0}^{k-1}[I^j\otimes \bar{Y}_k].$$
    Hence, we get the desired result by canceling common factors.
\end{proof}

\begin{lemma}\label{lemma2}
Let $Y\in\mathcal{Y}$. Then $\mathscr{R}$ restricts to a function $\mathscr{F}^{-1}Y\to\mathscr{F}^{-1}\{\mathscr{S}Y\}$, each fibre of which contains $$|\Hom(Y_0,\bar{Y}_0)|^{-1}\prod_{j=0}^\infty|\Hom(Y_0,I^j\otimes \bar{Y}_j)|$$ elements.
\end{lemma}
\begin{proof}
This restricted function is well defined because $\mathscr{F}\mathscr{R}=\mathscr{S}\mathscr{F}$. Let $X'\in \mathscr{F}^{-1}\{\mathscr{S}Y\}$. Then $X$ is in the fibre of $X'$ if and only if $X$ is a finite colength submodule of $M$ satisfying $X'=M\cap I^{-1}X$ and $\frac{X+IM}{IM}=Y_0$. In fact, from the calculation
\begin{align*}
    \length(M/X)&=\length\left(\frac{M}{X+IM}\right)+\length\left(\frac{X+IM}{X}\right)\\ &=\length\left(\frac{M/IM}{(X+IM)/IM}\right)+\length\left(I\otimes \frac{M}{M\cap I^{-1}X}\right),
\end{align*}
we see that the finite colength condition is superfluous. Moreover, from the condition $X'=M\cap I^{-1}X$, we see that if $X$ is in the fibre of $X'$, then $X$ is a submodule of $X'$ containing $IX'$. Therefore, $X\mapsto X/IX'$ gives an injection from the fibre of $X'$ to $\{\text{submodules of $X'/IX'$}\}$, and so it suffices to count its image $\mathcal{I}$.

To this end, we compute that $$\frac{X'/IX'}{(X'\cap IM)/IX'}\cong\frac{X'}{X'\cap IM}\cong\frac{X'+IM}{IM}=Y_1.$$
Therefore, we have a canonical short exact sequence of left $R/I$-modules,
$$0\to\frac{X'\cap IM}{IX'}\to\frac{X'}{IX'}\to Y_1\to 0.$$
Now, let $X$ be an arbitrary submodule of $X'$ containing $IX'$, so that $\frac{X}{IX'}$ is a submodule of the middle term. Its preimage on the left is equal to $\frac{X\cap IM}{IX'}$, and its image on the right is equal to $\frac{X+IM}{IM}$. Hence, $\mathcal{I}$ consists of those submodules of $\frac{X'}{IX'}$ with trivial preimage on the left and image $Y_0$ on the right. Let $\iota:Y_0\hookrightarrow Y_1$ be the canonical embedding. Then $\tilde{\iota}\mapsto\im(\tilde{\iota})$ defines a bijection from $\{\text{lifts of $\iota$}\}$ to $\mathcal{I}$, and so it suffices to count $\{\text{lifts of $\iota$}\}$.
\[\begin{tikzcd}
	0 & {\frac{X'\cap IM}{IX'}} & {\frac{X'}{IX'}} & {Y_1} & 0 \\
	&&& {Y_0}
	\arrow[from=1-1, to=1-2]
	\arrow[from=1-2, to=1-3]
	\arrow[from=1-3, to=1-4]
	\arrow[from=1-4, to=1-5]
	\arrow["{\tilde{\iota}}", dotted, from=2-4, to=1-3]
	\arrow["\iota"', hook', from=2-4, to=1-4]
\end{tikzcd}\]

Since $Y_0$ is a finite colength submodule of $M/IM$, it is projective over $R/I$ by assumption. Applying the exact functor $\Hom(Y_0,-)$ to our short exact sequence, we see that $|\{\text{lifts of $\iota$}\}|=|\Hom(Y_0,\frac{X'\cap IM}{IX'})|$. Note that $|\Hom(Y_0,-)|$ defines a function from the set of finite-length left $R/I$-modules to $\mathbb{N}$, which descends to a monoid homomorphism $G\to\mathbb{N}$, which, in turn, extends to a group homomorphism $\tilde{G}\to\mathbb{Q}^\times$. So, we compute that 
\begin{align*}
    \left[\frac{X'\cap IM}{IX'}\right]&=\left[\frac{IM}{IX'}\right]\left[\frac{IM}{X'\cap IM}\right]^{-1}=\left[I\otimes\frac{M}{X'}\right]\left[\frac{X'+IM}{X'}\right]^{-1}=\left[I\otimes\frac{M}{X'}\right]\left[\frac{M}{X'}\right]^{-1}\left[\frac{M}{X'+IM}\right]
\end{align*}
in $\tilde{G}$. Applying Proposition \ref{prop_graded} and canceling common factors (like in the proof of Lemma \ref{lemma1}), we find that 
$$\left[\frac{X'\cap IM}{IX'}\right]=[\bar{Y}_0]^{-1}\prod_{j=0}^\infty[I^j\otimes \bar{Y}_j].$$
Hence, 
$$|\{\text{lifts of $\iota$}\}|=\left|\Hom\left(Y_0,\frac{X'\cap IM}{IX'}\right)\right|=|\Hom(Y_0,\bar{Y}_0)|^{-1}\prod_{j=0}^\infty|\Hom(Y_0,I^j\otimes \bar{Y}_j)|,$$
thereby completing the proof of our second lemma.
\end{proof}

\begin{proposition}\label{prop_fundamental}
We have
\begin{align}\label{fundamental_formula}
   \sum_{X\in\mathscr{F}^{-1}Y}[M/X]=\prod_{j=0}^{\infty}|\Hom(Y_j,\bar{Y}_j)|^{-1}\prod_{k=0}^j |\Hom(Y_{j-k},I^{k}\otimes \bar{Y}_j)|[I^k\otimes \bar{Y}_j]\in\Dir 
\end{align}
for each $Y\in\mathcal{Y}$.
\end{proposition}
\begin{proof}
Note that $Y_n=M/IM$ for some $n$. We proceed by induction on $n$.

When $n=0$, Nakayama's lemma tells us that $\mathscr{F}^{-1}Y=\{M\}$, and so $\sum_{X\in\mathscr{F}^{-1}Y}[M/X]=[0]=1\in\Dir$. Let's now assume the induction hypothesis:
\begin{align}\label{induction_hypothesis}
    \text{formula (\ref{fundamental_formula}) holds for each $Y\in\mathcal{Y}$ with $Y_n=M/IM$.}
\end{align}
Then we take any $Y\in\mathcal{Y}$ with $Y_{n+1}=M/IM$. By Lemma \ref{lemma1},
\begin{align*}
    \sum_{X\in\mathscr{F}^{-1}Y}[M/X]=\sum_{X\in\mathscr{F}^{-1}Y}[M/\mathscr{R}X][\mathscr{R}X/X]=\left(\prod_{k=0}^{\infty}[I^k\otimes \bar{Y}_k]\right)\sum_{X\in\mathscr{F}^{-1}Y} [M/\mathscr{R}X].
\end{align*}
By Lemma \ref{lemma2},
$$\sum_{X\in\mathscr{F}^{-1}Y}[M/\mathscr{R}X]=\left(|\Hom(Y_0,\bar{Y}_0)|^{-1}\prod_{j=0}^\infty|\Hom(Y_0,I^j\otimes \bar{Y}_j)|\right)\sum_{X'\in\mathscr{F}^{-1}\{\mathscr{S}Y\}}[M/X'].$$
By the induction hypothesis (\ref{induction_hypothesis}), 
\begin{align*}
    \sum_{X'\in\mathscr{F}^{-1}\{\mathscr{S}Y\}}[M/X']&=\prod_{j=0}^{\infty} |\Hom(Y_{j+1},\bar{Y}_{j+1})|^{-1}\prod_{k=0}^j |\Hom(Y_{j-k+1},I^{k}\otimes \bar{Y}_{j+1})|[I^k\otimes \bar{Y}_{j+1}]\\
    &=\prod_{j=1}^{\infty} |\Hom(Y_{j},\bar{Y}_{j})|^{-1}\prod_{k=0}^{j-1} |\Hom(Y_{j-k},I^{k}\otimes \bar{Y}_{j})|[I^k\otimes \bar{Y}_{j}].
\end{align*}
Combining these three intermediate calculations, we complete the proof of our proposition.
\end{proof}

For finitely generated left $R/I$-modules $A,B,C$ with $B$ of finite length, we let $F_{BC}^A$ be the number of submodules $D$ of $A$ such that $D\simeq C$ and $A/D\simeq B$. This is a \textit{Hall number}, and it depends only on the isomorphism classes of $A,B,C$ (cf. \cite{iyama_hall}). These numbers are related to the partial zeta function of $A$ and $C$ via the formula
$$Z(A,C)=\sum_B F_{BC}^A [B] \in \Dir $$
where the sum is over all finite length left $R/I$-modules $B$ up to isomorphism (cf. \cite{iyama_rep}). Our assumption that $R$ is left arithmetical guarantees that these Hall numbers are all finite, and that the set of isomorphism classes of finite length left $R/I$-modules is countable.

Moreover, these Hall numbers may be multiplied to count chains with certain isomorphism class restrictions. Namely, if $(A_j)$ is a sequence of finitely generated left $R/I$-modules that eventually stabilises to $A$, and $(B_j)$ is a sequence of finite length left $R/I$-modules that eventually stabilises to $0$, then 
$$\prod_{j=0}^\infty F_{B_j A_j}^{A_{j+1}}$$
counts the number of chains $C_0\subseteq \cdots \subseteq C_j \subseteq \cdots =A$ of finite colength submodules of $A$ that eventually stabilise to $A$ and satisfy $C_j\simeq A_j$ and $C_{j+1}/C_j\simeq B_j$ for each $j$ (cf. \cite[Remark~1.3(v), p.~8]{schiffmann_lectures} or \cite[Eqn.~(1.5), p.~26]{hubery_notes}).

For a sequence $P$ of finitely generated projective left $R/I$-modules and a nonnegative integer $j$, we have a function from the set of finite length left $R/I$-modules to $\Dir$ given by
$$\langle P,-\rangle_j^I=|\Hom(P_j,-)|^{-1}\prod_{k=0}^j |\Hom(P_{j-k},I^{k}\otimes -)|[I^k\otimes -]$$ (cf. formula (\ref{fundamental_formula})). It descends to a monoid homomorphism $G\to(\Dir,\cdot)$ whose continuous $\mathbb{C}$-linear extension is an endomorphism of $\Dir$, thereby giving a change of variable for our partial zeta functions in $\Dir$. Note that $\langle P,-\rangle_j^I$ depends on $P$ up to entrywise isomorphism.

In preparation for our next and final formula, we recall or introduce the following data. 
\begin{itemize}
    \item Our map $\mathscr{F}:\mathcal{X}\to\mathcal{Y}$ from before.
    \item Let $\mathcal{P}=\mathcal{Y}/\sim$ where $Y\sim Y'$ if and only if $Y_j\simeq Y_j'$ for each $j$. Then we have a canonical map $\mathcal{Y}\to\mathcal{P}$. We may represent each $P\in\mathcal{P}$ as a sequence $(P_j)$ of isomorphism classes of finitely generated projective left $R/I$-modules such that $P_j=[M/IM]$ for $j\gg0$.
    \item Let $\mathcal{Z}$ be the set of sequences $(Z_j)$ of isomorphism classes of finite length left $R/I$-modules such that $Z_j=0$ for $j\gg0$. Then $Y\mapsto([\bar{Y}_j])$ defines a map $\mathcal{Y}\to\mathcal{Z}$.
    \item The product of the maps $\mathcal{Y}\to\mathcal{P}$ and $\mathcal{Y}\to\mathcal{Z}$ gives us a map $\mathscr{H}:\mathcal{Y}\to\mathcal{P}\times\mathcal{Z}$. We also have the canonical projection $\mathcal{P}\times\mathcal{Z}\to\mathcal{P}$, which we compose with $\mathscr{H}$ and $\mathscr{F}$ to get another map $\mathscr{G}:\mathcal{X}\to\mathcal{Y}\to\mathcal{P}\times\mathcal{Z}\to\mathcal{P}$.
\end{itemize}
 We illustrate this in the commutative diagram
\[\begin{tikzcd}
	& {\mathcal{X}} \\
	& {\mathcal{Y}} \\
	{\mathcal{P}} & {\mathcal{P}\times\mathcal{Z}.}
	\arrow["{\mathscr{F}}", from=1-2, to=2-2]
	\arrow["{\mathscr{G}}"', dashed, from=1-2, to=3-1]
	\arrow["{\mathscr{H}}", from=2-2, to=3-2]
	\arrow[from=3-2, to=3-1]
\end{tikzcd}\]

\begin{theorem}\label{thm_prolif}
    Let $R$ be a left arithmetical ring, and let $M$ be a finitely generated projective left $R$-module. Suppose that there exists an invertible ideal $I$ of $R$, contained in the Jacobson radical of $R$, such that the finite colength submodules of $M/IM$ are all projective over $R/I$. Then
    \begin{align*}
        Z(M)=\Prolif(M/IM;I\otimes-),
    \end{align*}
where
\begin{itemize}
    \item $\Prolif(M/IM;I\otimes-)=\sum_{P\in\mathcal{P}}\prod_{j=0}^{\infty}Z(P_{j+1},P_j;\langle P,-\rangle^{I}_j)$;
    \item $\mathcal{P}$ is the set of sequences $P=(P_j)$ of isomorphism classes of finitely generated projective left $R/I$-modules of the form $([Y_j])$, where $Y_0\subseteq Y_1\subseteq\cdots$ is a chain of finite colength submodules of $M/IM$ eventually stabilising to $M/IM$;
    \item $\langle P,-\rangle^I_j=|\Hom(P_j,-)|^{-1}\prod_{k=0}^j |\Hom(P_{j-k},I^{k}\otimes -)|[I^{k}\otimes -]$.
\end{itemize}
We refer to $\Prolif(M/IM;I\otimes-)$ as the proliferation of $M/IM$ by $I\otimes-$. We may similarly consider proliferations by other monoid automorphisms of $G$.
\end{theorem}
\begin{proof}
Firstly, we have
$$Z(M)=\sum_{X\in\mathcal{X}}[M/X]=\sum_{P\in\mathcal{P}}\ \sum_{X\in\mathscr{G}^{-1}P}[M/X].$$
So, we're left to study the inner-sum for each $P\in\mathcal{P}$. To this end, we have
$$\sum_{X\in\mathscr{G}^{-1}P}[M/X]=\sum_{Z\in\mathcal{Z}}\ \sum_{Y\in\mathscr{H}^{-1}(P,Z)}\ \sum_{X\in\mathscr{F}^{-1}Y}[M/X].$$
By Proposition \ref{prop_fundamental}, the right-hand side is equal to 
\begin{align*}
    \sum_{Z\in\mathcal{Z}}\ \sum_{Y\in\mathscr{H}^{-1}(P,Z)}\ \prod_{j=0}^{\infty}\langle P,Z_j\rangle^I_j.
\end{align*}
The inner-sum consists of
$$\prod_{j=0}^{\infty}F_{Z_j P_j}^{P_{j+1}}$$
terms. Hence, our double sum is equal to 
\begin{align*}
    \sum_{Z\in\mathcal{Z}}\prod_{j=0}^{\infty}F_{Z_j P_j}^{P_{j+1}}\langle P,Z_j\rangle^I_j=\prod_{j=0}^{\infty}\sum_{Z_j}F_{Z_j P_j}^{P_{j+1}}\langle P,Z_j\rangle^I_j=\prod_{j=0}^{\infty}Z(P_{j+1},P_j;\langle P,-\rangle^I_j),
\end{align*}
where $\sum_{Z_j}$ is over all finite length left $R/I$-modules up to isomorphism.
\end{proof}

\begin{corollary}\label{corollary_single_sliver}
 Let $R$ be a left arithmetical ring, and let $M$ be a finitely generated projective left $R$-module. Suppose that there exists an invertible ideal $I$ of $R$, contained in the Jacobson radical of $R$, such that the finite colength submodules of $M/IM$ are all isomorphic. Then
    $$Z(M)=\Prolif(M/IM,I\otimes-)=\prod_{j=0}^\infty Z(M/IM;\langle M/IM,-\rangle^I_j),$$
    where each
    $$\langle M/IM,-\rangle^I_j=|\Hom(M/IM,-)|^{-1}\prod_{k=0}^j |\Hom(M/IM,I^{k}\otimes -)|[I^{k}\otimes -].$$
\end{corollary}
\begin{proof}
    This is a special case of Theorem \ref{thm_prolif}, when $\mathcal{P}$ is the singleton set consisting of the sequence that is constantly $[M/IM]$.
\end{proof}

\begin{example}\label{example_voll}
    Let $R$ be a ring, with invertible Jacobson radical $J$, such that $R/J$ is finite. Let $M$ be a finitely generated projective left $R$-module. Then we may apply Theorem \ref{thm_prolif}. Since $R/J$ is a finite semisimple ring, we may explicitly write the partial zeta functions over $R/J$ in terms of $q$-binomial coefficients. In the special case when $R$ is a commutative complete discrete valuation ring with finite residue field, this application of Theorem \ref{thm_prolif} implies a well-known formula \cite[E.g.~2.20]{Voll}.
\end{example}

\begin{example}\label{example_rossmann}
    Take $(R,I,M)=(\mathbb{Z}\llbracket t\rrbracket, (t), {}_{\mathbb{Z}\llbracket t\rrbracket} \mathbb{Z}\llbracket t\rrbracket)$. Then $R/I\cong\mathbb{Z}$, and so we may identify $\Dir(R)$ with the ordinary Dirichlet algebra. Corollary \ref{corollary_single_sliver} implies Rossmann's formula \cite[Thm.~4.4]{rossmann}
    $$\zeta({}_R R;s)=\prod_{j=1}^\infty\zeta(js-j+1),$$
    where $\zeta(s)$ is the Riemann zeta function. 
    
    Alternatively, we may arrive at Rossmann's formula via the local-global method of \cite{chan2024zetafunctionsorderssurfaces}. Note that $T=\Spec(R)$ is a two-dimensional regular scheme with well-defined \textit{arithmetic zeta function} $\zeta(T,s)=\zeta(s)$ \cite{Serre}. Hence, Rossmann's formula follows easily from the Chinese remainder theorem and Lustig's local formula \cite{Lus56} (see also Example \ref{example_lustig}).
\end{example}

\section{Lifted Hey formula}\label{section_lift_hey}

Let $R$ be a ring with Jacobson radical $J$ such that $R/J$ is finite and ${}_R J$ is finitely generated. Let $M$ be a finitely generated projective left $R$-module. Suppose that there exists an invertible ideal $I\subseteq J$ of $R$ such that the finite colength submodules of $M/IM$ are all isomorphic. Write out Artin-Wedderburn decompositions
$$R/J\simeq\prod_{i\in\mathcal{S}} \M_{r_i}(\mathbb{F}_{q_i}) \qquad \mbox{and} \qquad M/JM\simeq \bigoplus_{i\in\mathcal{S}} \left(\mathbb{F}_{q_i}^{\oplus r_i}\right)^{\oplus m_i}$$
over $\mathcal{S}=\mathcal{S}(R)$. Write $S_i=\mathbb{F}_{q_i}^{\oplus r_i}$ and $z_i=[S_i]\in G(R)$ for each $i$. Therefore, $\Dir(R)=\mathbb{C}\llbracket z_i\rrbracket_{i\in\mathcal{S}}$. Let $\sigma$ be the permutation on $\mathcal{S}$ defined by $I\otimes_R S_i\simeq S_{\sigma(i)}$ for each $i$.

\begin{theorem}\label{thm_lifted_hey}
    We have
    \begin{align}\label{W_formula}
        Z(M;z)=\Prolif(M/IM,I\otimes-)=\prod_{n=0}^\infty\prod_{i\in\mathcal{S}}\prod_{j=0}^{m_i-1}(1-q_i^{j-m_i}\prod_{k=0}^n w_{\sigma^k(i)})^{-1},
    \end{align}
     where each $w_i=q_i^{m_i}z_i$.
\end{theorem}
\begin{proof}
    By Corollary \ref{corollary_single_sliver}, 
    $$Z(M)=\Prolif(M/IM,I\otimes-)=\prod_{n=0}^\infty Z(M/IM;\langle M/IM,-\rangle^I_n),$$
    where each
    $$\langle M/IM,-\rangle^I_n=|\Hom_{R}(M/IM,-)|^{-1}\prod_{k=0}^n |\Hom_{R}(M/IM,I^{k}\otimes_R -)|[I^{k}\otimes_R -].$$
    By our abstract Hey formula (see Appendix \ref{section_hey}),
    $$Z(M/IM)=\prod_{i\in\mathcal{S}}\prod_{j=0}^{m_i-1}(1-q_i^j z_i)^{-1}.$$
    So, we're left to compute $\langle M/IM,S_i\rangle^I_n$ for each $i$ and $n$. First, we compute that
    \begin{align*}
        |\Hom_R(M/IM,S_i)|&=|\Hom_{R/J}(M/JM,S_i)|\\
        &=|\Hom_{\M_{r_i}(\mathbb{F}_{q_i})}(\left(\mathbb{F}_{q_i}^{\oplus r_i}\right)^{\oplus m_i},\mathbb{F}_{q_i}^{\oplus r_i})|\\
        &=|\Hom_{\mathbb{F}_{q_i}}(\mathbb{F}_{q_i}^{\oplus m_i},\mathbb{F}_{q_i})|\\
        &=q_i^{m_i}
    \end{align*}
    for each $i$. Second, we see that $[I^k\otimes_R S_i]=[S_{\sigma^k(i)}]=z_{\sigma^k(i)}$ for each $i$ and $k$. Therefore, we also have
    $$|\Hom_R(M/IM,I^k\otimes_R S_i)|=q_{\sigma^k(i)}^{m_{\sigma^k(i)}}$$
    for each $i$ and $k$. Combining these calculations, we arrive at our desired formula.
\end{proof}

Write $n^{-s}=[\mathbb{Z}/n\mathbb{Z}]\in G(\mathbb{Z})$ for each positive integer $n$. Then the canonical map $G(R)\to G(\mathbb{Z})$ takes $z_i\mapsto q_i^{-r_is}$ for each $i\in\mathcal{S}(R)$. Moreover, the induced map $\Dir(R)\to \Dir(\mathbb{Z})$ takes the Bushnell-Reiner zeta function $Z(M;z)$ to the familiar Solomon zeta function $\zeta(M;s)$.

The following special case of Theorem \ref{thm_lifted_hey} has pleasant global applications, to zeta functions enumerating (possibly noncommutative analogues of) zero-dimensional closed subschemes of an arithmetic surface $X$. This is demonstrated in \cite{chan2024zetafunctionsorderssurfaces} by interchanging the Euler product over closed points of $X$ with the infinite product that we see below, so as to compare Solomon zeta functions with arithmetic zeta functions \cite{Serre}.

\begin{corollary}\label{cor_hom_slice}
If the simple components of $R/J$ are all isomorphic and $m_i=m$ for every $i\in\mathcal{S}$ (that is, every simple left $R$-module has the same multiplicity $m$ in $M/JM$), then we have 
$$\zeta(M;s)=\prod_{n=0}^\infty \prod_{j=0}^{m-1}(1-q^{j+mn-r(n+1)s})^{-|\mathcal{S}|}.$$
Here, each $q_i=q$ and $r_i=r$.
\end{corollary}

\begin{example}\label{example_lustig}
    Let $R$ be a commutative two-dimensional regular local ring with finite residue field $\mathbb{F}_q$. Let $\mathfrak{m}$ be the maximal ideal of $R$. Take $I=(x)$, for any $x\in\mathfrak{m}-\mathfrak{m}^2$, so that $R/I$ is a commutative discrete valuation ring with finite residue field $\mathbb{F}_q$. Then Corollary \ref{cor_hom_slice} implies Lustig's formula \cite{Lus56} (see \cite{Ber03} for an English summary)
    \begin{align}\label{lustig_formula}
    \zeta({}_R R;s)=\prod_{n=0}^\infty (1-q^{n-(n+1)s})^{-1}.
\end{align}
Here, $r=m=|\mathcal{S}|=1$. For positive integers $i$ and $j$, let $p(i,j)$ be the number of integer partitions of $i$ with greatest part $j$. Then we recall Euler's two variable generating function \cite[\S.~2.1]{And84}
$$1+\sum_{i=1}^\infty\sum_{j=1}^i p(i,j) z^i w^j=\prod_{n=1}^\infty (1-wz^n)^{-1}.$$
Taking $z=q^{1-s}$ and $w=q^{-1}$, we see that
$$\sum_{j=1}^i p(i,j)q^{i-j}$$
enumerates the ideals of $R$ with colength $i\ge 1$.
\end{example}

Corollary \ref{cor_hom_slice} may be applied over many noncommutative two-dimensional orders \cite[Proposition~3.17]{chan2024zetafunctionsorderssurfaces}.

\section{Two-dimensional regular semiperfect rings}\label{section_rump}

We begin with a definition from \cite{rump_reg}.

\begin{definition}
   A (left) \textit{two-dimensional regular semiperfect ring} is a left Noetherian semiperfect ring $R$, with Jacobson radical $J$, in which there exist two invertible ideals $I_1,I_2$ satisfying $I_1+I_2=J$ and $I_1I_2=I_1\cap I_2=I_2I_1$. Such a ring $R$ is semilocal with finitely generated Jacobson radical $J$ (as a left and right ideal). Therefore, by Proposition \ref{prop_slzr}, $R$ is arithmetical if and only if $R/J$ is finite. We always assume that this is the case, so that we can apply Theorem \ref{thm_prolif}. Moreover, we call $I_1,I_2$ a \textit{regular system of ideals} for $R$ (it is not unique).    
\end{definition}

Rump classified two-dimensional regular semiperfect rings in terms of rank two almost regular valuations \cite[Theorem~4.1]{rump_reg}, for which it is straightforward to generate a large number of examples. Furthermore, $R$ is a left order of left global dimension two in an (Artinian) semisimple ring \cite[Proposition~1.4 and Theorem~4.2]{rump_reg}.

Let $R$ be an arithmetical two-dimensional regular semiperfect ring with a regular system of ideals $I_1,I_2$. Choose $I\in\{I_1,I_2\}$. Then it follows that $R/I$ is a one-dimensional regular semiperfect ring (Definition \ref{definition_one-dim_regular_semiperfect}). As noted by Rump \cite[Proposition~1.6]{rump_reg}, one-dimensional regular semiperfect rings were classified by Michler \cite{Michler1969}, and they generalise classical hereditary orders \cite[\S.~39]{Rei03}. In particular, $R/I$ is a hereditary order in a semisimple ring $A$. Let $M$ be a finitely generated projective left $R$-module. By Theorem \ref{thm_prolif}, 
$$Z(M)=\Prolif(M/IM,I\otimes-).$$
We want to study such equations, and both sides are compatible with Morita equivalence in an obvious way. So we may assume, without loss of generality, that $R$ is basic and thus $R/I$ is basic as well. In Appendix \ref{appendix_brz_hereditary}, we give effective formulae for all the requisite partial Bushnell-Reiner zeta functions over $R/I$. So, let us now look at the set $\mathcal{P}$ and the changes of variable $\langle P,-\rangle^{I}_j$ (as defined in Theorem \ref{thm_prolif}).

For each $i\in\mathcal{S}(R/I)$, let $E_i\twoheadrightarrow S_i$ be a projective cover over $R/I$ for any $S_i\in i$. Then the $E_i$ give a complete set of representatives for the isomorphism classes of principal indecomposable left $R/I$-modules, and thus their images $[E_i]$ form a basis for the Grothendieck group $K_0(R/I)$. Let $K_0(R/I)^+\subset K_0(R/I)$ be the free commutative monoid generated by the $[E_i]$. 

Note that $\mathcal{S}(A)$ forms a basis for the Grothendieck group $K_0(A)$, and then take the associated free commutative monoid $K_0(A)^+\subset K_0(A)$. Since each $A\otimes_{R/I} E_i\in\mathcal{S}(A)$, the canonical group homomorphism $K_0(R/I)\to K_0(A)$ restricts to a monoid homomorphism $K_0(R/I)^+\to K_0(A)^+$. 

We have the finitely generated left $A$-module $V=A\otimes_{R/I}M/IM$. Then, for any finite colength submodule $Y_j\subseteq M/IM$, the canonical embedding $Y_j\hookrightarrow M/IM$ induces an isomorphism $A\otimes_{R/I} Y_j\xrightarrow{\sim} V$ of left $A$-modules. So, the monoid homomorphism $K_0(R/I)^+\to K_0(A)^+$ takes each $[Y_j]\mapsto[V]$. Conversely, for any finite colength submodule $Y_{j+1}\subseteq M/IM$ and any $c\in K_0(R/I)^+$ in the fibre of $[V]$, there exists a finite colength submodule $Y_j\subseteq Y_{j+1}$ with $c=[Y_j]\in K_0(R/I)^+$. 

Let $\mathcal{K}\subseteq K_0(R/I)^+$ be the fibre of $[V]$. Then we may identify $\mathcal{P}$ with the set of sequences of classes in $\mathcal{K}$ that eventually stabilise to $[M/IM]$.

Moreover, we see that $|\Hom_{R/I}(-,-)|$ defines a pairing $K_0(R/I)\times \tilde{G}(R/I)\to\mathbb{Q}^\times$ that is a group homomorphism in each variable separately. It is completely determined by the values $$|\Hom_{R/I}(E_i,S_j)|=|\Hom_{R/J}(S_i,S_j)|=|\End_{R/J}(S_i)|^{\delta_{ij}}$$ 
for $i,j\in\mathcal{S}(R/I)$, where $\delta_{ij}$ is the Kronecker delta.

An interesting feature of our approach is that we get to \textit{choose} $I$. We address a consequence of this in the following example.

\begin{example}
Let $\bar{R}$ be a basic (left) one-dimensional regular semiperfect ring, with Jacobson radical $\bar{J}$, such that $\bar{R}/\bar{J}$ is finite. Then $\bar{J}$ is principally generated by a normal regular element $g\in\bar{R}$, and $\bar{R}/\bar{J}$ is a finite direct product of finite fields. Let $R=\bar{R}\llbracket t\rrbracket$ be the univariate formal power series ring. Then $R$ is a left Noetherian semiperfect ring with Jacobson radical $J=g\bar{R}+tR$. In fact, $I_1=tR$ and $I_2=gR$ are invertible ideals of $R$ satisfying $I_1+I_2=J$ and $I_1I_2=I_1\cap I_2=I_2I_1$. Hence, $R$ is a left two-dimensional regular semiperfect ring. Let $\bar{M}$ be a finitely generated projective left $\bar{R}$-module, and let $M$ be a projective cover of $\bar{M}$ over $R$. Then 
\begin{align*}
    \Prolif(M/I_1M,I_1\otimes-)=Z(M)=\Prolif(M/I_2M,I_2\otimes-).
\end{align*}
On the one hand, $M/I_1M=M/tM\cong\bar{M}$ and $I_1\otimes_{R}-$ acts trivially on $\mathcal{S}$. Therefore,
$$\Prolif(M/I_1M,I_1\otimes-)=\Prolif(\bar{M},\id).$$
On the other hand, $R/I_2\cong(\bar{R}/\bar{J})\llbracket t\rrbracket$. Therefore, $R/I_2$ is isomorphic to a finite direct product of univariate formal power series rings over finite fields. So, the finite colength submodules of $M/I_2M$ are all isomorphic, and hence our lifted Hey formula (Theorem \ref{thm_lifted_hey}) expresses $\Prolif(\bar{M},\id)=\Prolif(M/I_2M,I_2\otimes-)$ as an infinite product.
\end{example}

For the remainder of this section, we further assume that $R/I$ is indecomposable. First, let us consider the partial zeta functions $Z(P_{j+1},P_j)$ over $R/I$ appearing in $\Prolif(M/IM,I\otimes-)$. Put $v=\prod_{i\in\mathcal{S}}[S_i]\in G$, $\ell=\length_A(V)$, and then
$$Z(V)=\prod_{j=0}^{\ell-1}(1-q^j v)^{-1}\in\Dir.$$
By the Bushnell-Reiner-Solomon theorem (see \S.~\ref{BRS_factor}), each $$Z(P_{j+1},P_j)=Z(V)\cdot F(P_{j+1},P_j)$$ for some ``polynomial" $F(P_{j+1},P_j)\in\mathbb{C}[G]\subset\Dir$. 

Note that $I\otimes-$ fixes $v=\prod_{i\in\mathcal{S}}[S_i]\in G$. Moreover, since $R/I$ is indecomposable, there exists a finite field $\mathbb{F}_q$ such that $\End_{R/J}(S_i)\simeq\mathbb{F}_q$ for each $i\in\mathcal{S}$. Then we compute that 
$$|\Hom_{R/I}(E_i,v)|=\prod_{j\in\mathcal{S}}|\Hom_{R/I}(E_i,S_j)|=\prod_{j\in\mathcal{S}}|\End_{R/J}(S_i)|^{\delta_{ij}}=\prod_{j\in\mathcal{S}}q^{\delta_{ij}}=q$$
for each $i\in\mathcal{S}$. Therefore, $|\Hom_{R/I}(P_j,v)|=q^\ell$ for each $P_j$ in the fibre $\mathcal{K}\subset K_0(R/I)^+$ of $[V]\in K_0(A)^+$.
Hence, each $Z(V;\langle P,-\rangle_j^I)=Z_j(V)$ is independent of $P$ and $I$. 

Putting everything together, we have 
$$Z(M)=\Prolif(M/IM,I\otimes-)=\sum_{P\in\mathcal{P}}\prod_{j=0}^\infty Z(P_{j+1},P_j;\langle P,-\rangle^I_j)=\left(\prod_{j=0}^\infty Z_j(V)\right)\sum_{P\in\mathcal{P}}\prod_{j=0}^\infty F(P_{j+1},P_j;\langle P,-\rangle^I_j).$$

\appendix   

\section{Abstract Hey formula}\label{section_hey}

We prove the following abstract version of Hey's formula \cite{Hey27} (cf. \cite{Bushnell1987}) by adapting Solomon's method \cite[\S.~2--3]{Sol_77_first}. 

\begin{proposition}
    Let $R$ be a ring, with Jacobson radical $J$, such that $R/J$ is finite and ${}_R J$ is finitely generated. Let $M$ be a finitely generated left $R$-module. Suppose that the finite colength submodules of $M$ are all isomorphic. Write out Artin-Wedderburn decompositions
$$R/J\simeq\prod_{i} \M_{r_i}(\mathbb{F}_{q_i}) \qquad \mbox{and} \qquad M/JM\simeq \bigoplus_{i} \left(\mathbb{F}_{q_i}^{\oplus r_i}\right)^{\oplus m_i},$$
and put $z_i=[\mathbb{F}_{q_i}^{\oplus r_i}]$. Then 
\begin{align*}
    Z(M)=\prod_i\prod_{j=0}^{m_i-1}(1-q_i^j z_i)^{-1}.
\end{align*}
\end{proposition}
\begin{proof}
Let $\mathfrak{M}$ be the set of maximal submodules of $M$. For each nonempty $\mathfrak{S}\subseteq \mathfrak{M}$, let $\mathfrak{m}_\mathfrak{S}=\bigcap_{\mathfrak{m}\in \mathfrak{S}}\mathfrak{m}$. Then it follows from the inclusion-exclusion principle that
$$Z(M)=1+\sum_{\varnothing\ne \mathfrak{S}\subseteq \mathfrak{M}}(-1)^{|\mathfrak{S}|+1}[M/\mathfrak{m}_\mathfrak{S}]Z(\mathfrak{m}_\mathfrak{S}).$$
Putting $\mathfrak{m}_{\varnothing}=M$, we may rewrite this as
$$\sum_{\mathfrak{S}\subseteq \mathfrak{M}}(-1)^{|\mathfrak{S}|}[M/\mathfrak{m}_\mathfrak{S}]Z(\mathfrak{m}_\mathfrak{S})=1.$$
Now we observe that $\mathfrak{m}_\mathfrak{S}\simeq M$ and thus $Z(\mathfrak{m}_\mathfrak{S})=Z(M)$, for all $\mathfrak{S}\subseteq\mathfrak{M}$. Therefore,
\begin{align}\label{inverse_formula_1}
    Z(M)^{-1}=\sum_{\mathfrak{S}\subseteq \mathfrak{M}}(-1)^{|\mathfrak{S}|}[M/\mathfrak{m}_\mathfrak{S}].
\end{align}
Let $\mu$ be the M\"obius function on the finite poset $\mathcal{L}$ of submodules of $M$ containing $JM=\rad(M)$. Then we verify from the recursive definition of $\mu$ \cite[\S.~3]{Rota} that 
$$\mu(X,M)=\sum_{\substack{\mathfrak{S}\subseteq \mathfrak{M}\\ \mathfrak{m}_\mathfrak{S}=X}}(-1)^{|\mathfrak{S}|}$$
for every $X\in\mathcal{L}$. So, we find that the right-hand side of (\ref{inverse_formula_1}) is equal to
\begin{align}\label{inverse_formula_2}
    \sum_{X\in\mathcal{L}} \mu(X,M)[M/X].
\end{align}
Now, from our Artin-Wedderburn decompositions, we have a poset isomorphism $\mathcal{L}\simeq\prod_i\mathcal{L}_i$ where $\mathcal{L}_i$ is the $\mathbb{F}_{q_i}$-subspace poset of $V_i=\mathbb{F}_{q_i}^{\oplus m_i}$. Let $\mu_i$ be the M\"obius function on $\mathcal{L}_i$. Thus, (\ref{inverse_formula_2}) is equal to 
\begin{align}\label{inverse_formula_3}
    \prod_i\sum_{X_i\in\mathcal{L}_i} \mu_i(X_i,V_i)z_i^{\dim(V_i/X_i)}.
\end{align}
By a result of Hall (see \cite[E.g.~2, \S.~5]{Rota}), 
$$\mu_i(X_i,V_i)=(-1)^d q_i^{d(d-1)/2}$$
for every $X_i\in\mathcal{L}_i$ with $\dim(V_i/X_i)=d$. Now, making use of $q$-binomial coefficients and Cauchy's $q$-binomial theorem \cite[\S.~3.3]{And84}, we see that the $i$-th factor in (\ref{inverse_formula_3}) is equal to
$$\sum_{X_i\in\mathcal{L}_i} \mu_i(X_i,V_i)z_i^{\dim(V_i/X_i)}=\sum_{d=0}^{m_i} \binom{m_i}{d}_{q_i} (-1)^d q_i^{d(d-1)/2} z_i^d=\prod_{j=0}^{m_i-1}(1-q_i^j z_i).$$
\end{proof}

\section{Bushnell-Reiner zeta functions over hereditary orders}\label{appendix_brz_hereditary}

We begin with some definitions from \cite{rump_reg} and provide some examples.

\begin{definition}\label{definition_one-dim_regular_semiperfect}
    A (left) \textit{one-dimensional regular semiperfect ring} is a left Noetherian semiperfect ring $R$ with invertible Jacobson radical $J$. A \textit{noncommutative discrete valuation ring} is a left Noetherian domain $\Delta$ with a normal regular generator $\pi\in\Delta$ for its Jacobson radical $\mathfrak{J}=\Delta\pi=(\pi)=\pi\Delta$ such that the quotient $k=\Delta/\mathfrak{J}$ is a division ring. Such rings $R$ and $\Delta$ are semilocal with finitely generated Jacobson radicals (as left and right ideals). Therefore, by Proposition \ref{prop_slzr}, $R$ (resp. $\Delta)$ is arithmetical if and only if $R/J$ (resp. $\Delta/\mathfrak{J}$) is finite. We always assume that this is the case. Then, by Wedderburn's little theorem, $k$ is a finite field, and thus it is appropriate to call $k$ the \textit{residue field} of $\Delta$. We also call $\pi$ a \textit{uniformiser} for $\Delta$ (it is not unique). 
\end{definition}

\begin{example}\label{example_hereditary_orders}
    Let $\mathfrak{o}$ be the ring of integers in a nonarchimedean local field $K$. Let $\Delta$ be the integral closure of $\mathfrak{o}$ in a finite-dimensional central division $K$-algebra $D$. Then $\Delta$ is a noncommutative discrete valuation ring with finite residue field $k$. Let $\mathfrak{J}$ be the Jacobson radical of $\Delta$. Let $r\le n$ be positive integers. Let $(n_1,\ldots,n_r)$ be an $r$-tuple of positive integers such that $n_1+\cdots+n_r=n$. Then we have the one-dimensional regular semiperfect ring 
    $$R=\begin{pmatrix}\M_{n_1\times n_1}(\Delta) & \M_{n_1\times n_2}(\Delta) & \M_{n_1\times n_3}(\Delta) & \cdots & \M_{n_1\times n_r}(\Delta)\\
    \M_{n_2\times n_1}(\mathfrak{J}) & \M_{n_2\times n_2}(\Delta) & \M_{n_2\times n_3}(\Delta) & \cdots & \M_{n_2\times n_r}(\Delta)\\ \M_{n_3\times n_1}(\mathfrak{J}) & \M_{n_3\times n_2}(\mathfrak{J}) & \M_{n_3\times n_3}(\Delta) & \cdots & \M_{n_3\times n_r}(\Delta)\\ \vdots & \vdots & \vdots & \ddots & \vdots\\ \M_{n_r\times n_1}(\mathfrak{J}) & \M_{n_r\times n_2}(\mathfrak{J}) & \M_{n_r\times n_3}(\mathfrak{J}) & \cdots & \M_{n_r\times n_r}(\Delta) \end{pmatrix}\subseteq\M_{n\times n}(\Delta).$$
    Its Jacobson radical is
    $$J=\begin{pmatrix}\M_{n_1\times n_1}(\mathfrak{J}) & \M_{n_1\times n_2}(\Delta) & \M_{n_1\times n_3}(\Delta) & \cdots & \M_{n_1\times n_r}(\Delta)\\
    \M_{n_2\times n_1}(\mathfrak{J}) & \M_{n_2\times n_2}(\mathfrak{J}) & \M_{n_2\times n_3}(\Delta) & \cdots & \M_{n_2\times n_r}(\Delta)\\ \M_{n_3\times n_1}(\mathfrak{J}) & \M_{n_3\times n_2}(\mathfrak{J}) & \M_{n_3\times n_3}(\mathfrak{J}) & \cdots & \M_{n_3\times n_r}(\Delta)\\ \vdots & \vdots & \vdots & \ddots & \vdots\\ \M_{n_r\times n_1}(\mathfrak{J}) & \M_{n_r\times n_2}(\mathfrak{J}) & \M_{n_r\times n_3}(\mathfrak{J}) & \cdots & \M_{n_r\times n_r}(\mathfrak{J}) \end{pmatrix}\subseteq\M_{n\times n}(\Delta),$$
    with finite quotient
    $$R/J\simeq \M_{n_1\times n_1}(k)\times \cdots \times \M_{n_r\times n_r}(k),$$
    so that $R$ is arithmetical. In fact, $R$ is a (classical) hereditary $\mathfrak{o}$-order in $\M_{n\times n}(D)$ \cite[\S.~39]{Rei03}.
\end{example}

\begin{example}\label{example_morita_equivalence}
    Let $\Delta$, $\mathfrak{J}$, $k$ be as in Example \ref{example_hereditary_orders}. Let 
    $$R=\begin{pmatrix}
        \Delta & \Delta\\\mathfrak{J} & \Delta
    \end{pmatrix}, \qquad R'=\begin{pmatrix}\Delta & \Delta & \Delta\\ \Delta & \Delta & \Delta\\ \mathfrak{J} & \mathfrak{J} & \Delta\end{pmatrix}, \qquad P=\begin{pmatrix}\Delta & \Delta\\ \Delta & \Delta \\ \mathfrak{J} & \Delta\end{pmatrix}.$$
    Then $R$ and $R'$ are both arithmetical one-dimensional semiperfect regular rings, Morita equivalent via the invertible $(R',R)$-bimodule $P$. Let $J$ and $J'$ be the respective Jacobson radicals of $R$ and $R'$. Then we have $R/J\simeq k\times k$ and $R'/J'\simeq \M_{2\times 2}(k)\times k$. Thus, $R$ is basic but $R'$ is not. Now, let
    $$\mathfrak{m}_1=\begin{pmatrix}\mathfrak{J} & \Delta\\ \mathfrak{J} & \Delta\end{pmatrix} \qquad \mbox{and} \qquad \mathfrak{m}_2=\begin{pmatrix}\Delta & \Delta\\ \mathfrak{J} & \mathfrak{J}\end{pmatrix}.$$
    Then $\mathfrak{m}_1$ and $\mathfrak{m}_2$ are maximal left ideals of $R$, and $S_1=R/\mathfrak{m}_1$ and $S_2=R/\mathfrak{m}_2$ form a complete set of representatives of the isomorphism classes of simple left $R$-modules. We compute that
    $$P\otimes_R S_1\simeq P/P\mathfrak{m}_1\simeq\begin{pmatrix}k & 0\\ k & 0\\ 0 & 0\end{pmatrix} \qquad \mbox{and} \qquad P\otimes_R S_2\simeq P/P\mathfrak{m}_2\simeq\begin{pmatrix}0 & 0\\ 0 & 0\\ 0 & k\end{pmatrix}.$$
    Note that $|S_1|=|k|=|S_2|$ but $|P\otimes_R S_1|=|k|^2\ne|k|=|P\otimes_R S_2|$.
\end{example}

Let $R$ be an arithmetical one-dimensional semiperfect regular ring. We want to study partial Bushnell-Reiner zeta functions of finitely generated projective modules over $R$. By virtue of compatibility with Morita equivalence and direct products, it suffices to assume that $R$ is basic and indecomposable.

By Michler's classification theorem \cite{Michler1969} (also \cite[Proposition~1.6]{rump_reg}), we have a ring isomorphism
$$R\simeq\begin{pmatrix}\Delta & \Delta & \Delta & \cdots & \Delta\\
(\pi) & \Delta & \Delta & \cdots & \Delta\\
(\pi) & (\pi)& \Delta & \cdots & \Delta\\
\vdots & \vdots & \vdots & \ddots & \vdots\\
(\pi) & (\pi)& (\pi) & \cdots & \Delta
\end{pmatrix}$$
for some noncommutative discrete valuation ring $\Delta$ with uniformiser $\pi$ and finite residue field $k=\Delta/(\pi)$. Here, the right-hand side is the subring of $\M_{n\times n}(\Delta)$, consisting of those matrices that are upper-triangular modulo $(\pi)$, and the positive integer $n$ is equal to the number of isomorphism classes of simple left $R$-modules. From here on out, we identify $R$ with the right-hand side. Note that $R$ is a hereditary order in an Artinian simple ring $A$.

\subsection{Partial zeta functions}

We recall and build upon our framework in Section \ref{section_foundations}.

We have the free commutative monoid $G=G(R)$ on the set $\mathcal{S}=\mathcal{S}(R)$ of isomorphism classes of simple left $R$-modules, together with the norm $\mathscr{N}=\mathscr{N}_R$ on $G$ that takes each $i\in\mathcal{S}$ to the cardinality $|S_i|$ of any representative $S_i\in i$. We write $G$ multiplicatively. Then $(G,\mathscr{N})$ is an arithmetical semigroup, with associated Dirichlet algebra $\Dir(G,\mathscr{N})$. We typically denote this Dirichlet algebra by $\Dir(R)$ to use the left arithmetical ring functoriality of $\Dir$. Here, we want to use the arithmetical semigroup functoriality of $\Dir$, through which the aforementioned functoriality factors.

\begin{remark}\label{remark_finite_products}
   Suppose that we are given finitely many arithmetical semigroups $(G_i,\mathscr{N}_i)$, with corresponding generating sets $\mathcal{S}_i$. Then we may take the \textit{coproduct of arithmetical semigroups} $\coprod_i(G_i,\mathscr{N}_i)=(\coprod_i G_i,\coprod_i\mathscr{N}_i)$, consisting of the free commutative monoid $\coprod_i G_i$ on $\coprod_i \mathcal{S}_i$, together with the norm $\coprod_i\mathscr{N}_i$ that extends uniquely from the $\mathbf{Set}$ coproduct on $\coprod_i \mathcal{S}_i$. We have natural inclusion morphisms $(G_i,\mathscr{N}_i)\hookrightarrow\coprod_i(G_i,\mathscr{N}_i)$ into each constituent cofactor. Then we also have the Dirichlet algebra $\Dir(\coprod_i G_i,\coprod_i\mathscr{N}_i)$ and topological ring embeddings $\Dir(G_i,\mathscr{N}_i)\hookrightarrow\Dir(\coprod_i G_i,\coprod_i\mathscr{N}_i)$ induced by each inclusion morphism. In particular, when we have a finite direct product of left arithmetical rings $R=\prod_i R_i$, the canonical ring homomorphisms $R\to R_i$ induce the topological ring embeddings $\Dir(R_i)\hookrightarrow\Dir(R)$.
\end{remark}

For our application, we take a copy $(H,\mathscr{N})$ of $(G,\mathscr{N})=(G(R),\mathscr{N}_R)$. Then we write $z_i\in G$ and $w_i\in H$ for the images of $i\in\mathcal{S}$. Note that we have the identification $$\Dir(G\coprod H,\mathscr{N}\coprod\mathscr{N})=\mathbb{C}\llbracket z_i,w_j\rrbracket_{i,j\in\mathcal{S}}$$ of the associated Dirichlet algebra with a several variable formal power series ring. Moreover, in keeping with Bushnell and Reiner's original notation \cite{Bushnell1987}, we write $z^{\ell(F)}\in G$ and $w^{\ell(F)}\in H$ for the images of $[F]\in G(R)$ when $F$ is a finite length left $R$-module.

Let $J$ be the Jacobson radical of $R$ and $M$ be a finitely generated projective left $R$-module. Then we introduce a new zeta function
\begin{align}\label{refined_brz}
    Z(M;z,w)=\sum_{X\subseteq M} z^{\ell(M/X)}w^{\ell(X/JX)}\in\Dir(G\coprod H,\mathscr{N}\coprod\mathscr{N})
\end{align}
where the sum is over all finite colength submodules $X\subseteq M$. Since $R$ is arithmetical and hereditary and $M$ is finitely generated and projective, every $X$ is finitely generated and projective. Then, since $R$ is semilocal, the isomorphism class of $X$ is precisely determined by the composition factors of $X/JX$, i.e. by $w^{\ell(X/JX)}\in H$. Moreover, because there are only finitely many possibilities for $w^{\ell(X/JX)}$, the zeta function $Z(M;z,w)$ lies in the monoid subring $\Dir(G,\mathscr{N})[H]\subset \Dir(G\coprod H,\mathscr{N}\coprod\mathscr{N})$. For each $h\in H$, we have a natural projection $\Dir(G,\mathscr{N})[H]\to \Dir(G,\mathscr{N})$ given by the ``coefficient of $h$". The images of $Z(M;z,w)$ under these projections will be our partial Bushnell-Reiner zeta functions.

\subsection{Cyclic linearisation}

Let $\Delta$ be a noncommutative discrete valuation ring with uniformiser $\pi$ and finite residue field $k=\Delta/(\pi)\simeq\mathbb{F}_q$. Let $R$ be the subring of $\M_{n\times n}(\Delta)$ consisting of those matrices that are upper-triangular modulo $(\pi)$. Let $\Gamma\subseteq R$ be the subring of diagonal matrices. Let $C_1,\ldots,C_r$ be some columns of $R$, allowing for any number of repetitions but keeping their natural ordering. Let $M$ be the left $R$-submodule of $\M_{n\times r}(\Delta)$ whose columns are $C_1,\ldots,C_r$. Note that any finitely generated projective left $R$-module is isomorphic to one of this form. Then we let $M_1,\ldots,M_n$ be the rows of $M$. It follows that $$M_1\supseteq\cdots \supseteq M_n\supseteq\pi M_1.$$
For each $i\in\{1,2,\ldots,n\}$, we need copies $\Delta_i=\Delta$ and $k_i=k$. Then we have a commutative diagram of ring homomorphisms
\begin{center}
\begin{tikzcd}
&\Delta \arrow[hook]{r}\arrow[two heads]{d} &\Delta_1\times\cdots\times\Delta_n \arrow["\sim"]{r} \arrow[two heads]{d} & \Gamma \arrow[hook]{r} \arrow[two heads]{d} & R \arrow[two heads]{d} \\
&k\arrow[hook]{r} & k_1\times\cdots\times k_n \arrow["\sim"]{r} & \Gamma/\rad(\Gamma) \arrow["\sim"]{r} &  R/\rad(R).
\end{tikzcd}
\end{center}

\subsubsection{Pulling back to \texorpdfstring{$\Gamma$}{}}

Let $\mathcal{X}(\cdot)$ be the lattice of finite colength submodules, and let $\mathcal{S}(\cdot)$ be the set of isomorphism classes of simple left modules. Then we have a canonical lattice embedding
$$\gamma:\mathcal{X}({}_{R} M)\hookrightarrow\mathcal{X}({}_{\Gamma} M)$$
taking any $X\in \mathcal{X}({}_{R} M)$ to ${}_{\Gamma} X$. Moreover, we have the identification $\mathcal{S}(R)=\mathcal{S}(\Gamma)$ via the right-most square of our commutative diagram, so that we have 
$$z^{\ell({}_{R} M/X)}w^{\ell({}_{R} X/JX)}=z^{\ell({}_{\Gamma} M/X)}w^{\ell({}_{\Gamma} X/JX)}\in G\coprod H$$
for every $X\in \mathcal{X}({}_{R} M)$. We now want to find the image of $\gamma$ so that we can work with the formula 
$$Z(M;z,w)=\sum_{X\in\im(\gamma)} z^{\ell({}_{\Gamma} M/X)}w^{\ell({}_{\Gamma} X/JX)}.$$
To this end, let
$$g=\begin{pmatrix}0 & 1 & 0 & \cdots & 0\\ 0 & 0 & 1 & \cdots & 0\\ 0 & 0 & 0 & \cdots & 0\\ \vdots & \vdots & \vdots & \ddots & \vdots\\
0 & 0 & 0 & \cdots & 1\\
\pi & 0 & 0 & \cdots & 0\end{pmatrix}\in R$$ 
be the \textit{standard generator} for the Jacobson radical $J=\rad(R)$. Then we have the internal decomposition 
$$R=\bigoplus_{i=1}^{n}\Gamma g^{i-1}.$$ 
It follows that
$$\im(\gamma)=\{X\in\mathcal{X}({}_{\Gamma} M): gX\subseteq X\},$$
and hence we have the formula
$$Z(M;z,w)=\sum_{\substack{X\subseteq{}_{\Gamma} M\\ gX\subseteq X}} z^{\ell({}_{\Gamma} M/X)}w^{\ell({}_{\Gamma} X/JX)}.$$

\subsubsection{Decomposing over \texorpdfstring{$\Delta_1\times\cdots\times\Delta_n$}{}}

Next we have a lattice isomorphism
\begin{align*}
    \beta:\mathcal{X}({}_\Gamma M)\xrightarrow[]{\sim}\mathcal{X}({}_{\Delta_1} M_1)\times \cdots \times \mathcal{X}({}_{\Delta_n} M_n)
\end{align*}
that takes any finite colength submodule of ${}_\Gamma M$ to the $n$-tuple of its rows. Writing $$\beta(X)=(X_1,\ldots,X_n)\in \mathcal{X}({}_{\Delta_1} M_1)\times \cdots \times \mathcal{X}({}_{\Delta_n} M_n)$$ for any $X\in\mathcal{X}({}_\Gamma M)$, we have $$M/X\cong\bigoplus_{i=1}^n M_i/X_i$$ as left $\Delta_1\times\cdots\times\Delta_n$-modules. 

Let $(X_1,\ldots,X_n)\in \mathcal{X}({}_{\Delta_1} M_1)\times \cdots \times \mathcal{X}({}_{\Delta_n} M_n)$, and let $X$ be the $\Gamma$-submodule of $M$ with rows $X_1,\ldots,X_n$. Then $(X_1,\ldots,X_n)$ is in the image of $\beta\gamma$ if and only if $gX\subseteq X$. We compute that $gX$ is the $\Gamma$-submodule of $M$ with rows $X_2,\ldots,X_n,\pi X_1=X_{n+1}$. So the condition $gX\subseteq X$ is equivalent to $$X_1\supseteq \cdots \supseteq X_n\supseteq\pi X_1=X_{n+1}.$$ In this case, since $JX=gX\subseteq X$, we have
$$X/JX\cong \bigoplus_{i=1}^n X_i/X_{i+1}$$
as left $\Delta_1\times\cdots\times\Delta_n$-modules. 

Hence we have the formula
$$Z(M;z,w)=\sum_{X_*} \prod_{i=1}^n z_i^{\ell({}_{\Delta_i}M_i/X_i)}\cdot \prod_{j=1}^n w_j^{\ell({}_{\Delta_j}X_j/X_{j+1})},$$
where the sum is over all $(X_1,\ldots,X_n)\in\mathcal{X}({}_{\Delta_1}M_1)\times \cdots \times \mathcal{X}({}_{\Delta_n}M_n)$ such that $X_1\supseteq \cdots \supseteq X_n\supseteq\pi X_1=X_{n+1}$.

\subsubsection{Pulling back to \texorpdfstring{$\Delta$}{}}\label{pulling_back_delta}

 At this point, we see that we should first choose $X_1$, and then choose $(X_2,\ldots,X_n)$. Accordingly, we break $\sum_{X_*}$ into a double sum
$$\sum_{X_*}=\sum_{X_1}\sum_{(X_2,\ldots,X_n)}.$$
For compatibility with this double sum, we first write
\begin{align*}
    \prod_{i=1}^n z_i^{\ell(M_i/X_i)}=\prod_{i=1}^n z_i^{-\ell(M_1/M_i)+\ell(M_1/X_1)+\ell(X_1/X_i)}=\left(\prod_{i=1}^n z_i^{-\ell(M_1/M_i)}\right)\left(\prod_{i=1}^n z_i\right)^{\ell(M_1/X_1)}\left(\prod_{i=1}^n z_i^{\ell(X_1/X_i)}\right).
\end{align*}
And then we write
\begin{align*}
    \prod_{i=1}^n z_i^{\ell(X_1/X_i)}=\prod_{i=1}^n \prod_{j=1}^{i-1}z_i^{\ell(X_j/X_{j+1})}=\prod_{j=1}^{n-1}\left(\prod_{i=j+1}^{n}z_i\right)^{\ell(X_j/X_{j+1})}=\prod_{j=1}^{n}\left(\prod_{i=j+1}^{n}z_i\right)^{\ell(X_j/X_{j+1})},
\end{align*}
where, in the last equality, we take advantage of an empty product over $i$ when $j=n$.

Let
$$t_j=w_j\prod_{i=j+1}^n z_i, \qquad u=\prod_{i=1}^n z_i^{\ell(M_1/M_i)}, \qquad v=\prod_{i=1}^n z_i$$
for each $j\in\{1,\ldots,n\}$. Then we have
$$uZ(M;z,w)=\sum_{X_1}v^{\ell(M_1/X_1)}\sum_{(X_2,\ldots,X_n)}\prod_{j=1}^n t_j^{\ell(X_j/X_{j+1})},$$
where the first sum is over all $X_1\in\mathcal{X}(M_1)$, and the second sum is over all $(X_2,\ldots,X_n)\in\mathcal{X}(M_2)\times \cdots \times \mathcal{X}(M_n)$ such that $X_1\supseteq X_2\supseteq \cdots \supseteq X_n\supseteq\pi X_1=X_{n+1}$.

\subsection{Polynomials attached to finite filtered vector spaces}

Suppose that we are given a finite-dimensional filtered vector space $V_*:V=V_1\supseteq\cdots \supseteq V_{n+1}=0$ over $k$. Then we define the polynomial
\begin{align}\label{appendix_polynomial_filtered_vector_spaces}
    P(V_*;t)=\sum_{W_*} \prod_{j=1}^n t_j^{\dim(W_j/W_{j+1})},
\end{align}
where the sum is over all tuples of subspaces $W_*=(W_1,\ldots,W_{n+1})$ of $V$ satisfying 
\begin{center}
\begin{tikzcd}
& V_1 \arrow[symbol=\supseteq]{r} & V_2 \arrow[symbol=\supseteq]{r} & \cdots  \arrow[symbol=\supseteq]{r} &   V_n  \arrow[symbol=\supseteq]{r} & V_{n+1}\\
& W_1 \arrow[symbol=\supseteq]{r} \arrow[equal]{u} & W_2 \arrow[symbol=\supseteq]{r} \arrow[symbol=\subseteq]{u} & \cdots  \arrow[symbol=\supseteq]{r} &   W_n  \arrow[symbol=\supseteq]{r}  \arrow[symbol=\subseteq]{u} & W_{n+1} \arrow[equal]{u}.
\end{tikzcd}
\end{center}
Note the sum is really only over $(W_2,\ldots,W_n)$. If $U_*$ is a filtered vector space over $k$ that is isomorphic to $V_*$, then we easily verify that
$$P(U_*;t)=P(V_*;t).$$

Now let's prepare the data for our application. For each $X\in\mathcal{X}(M_1)$, let 
$$M_j(X)=(X\cap M_j)/\pi X$$
for each $j\in\{1,\ldots,n\}$ and consider the filtered vector space $M_*X:M_1(X)\supseteq \cdots \supseteq M_n(X)\supseteq M_{n+1}(X)=0$.
Then we have
\begin{align}\label{first_application_filtered}
    uZ(M;z,w)=\sum_{X_1}v^{\ell(M_1/X_1)}P(M_* X_1;t).
\end{align}

\begin{proposition}\label{proposition_vector_space_independence}
    Let $X_1\in\mathcal{X}(M_1)$, and let $Y_1=X_1+\pi M_1$. Then $M_* X_1\cong M_* Y_1$ as filtered vector spaces.
\end{proposition}
\begin{proof}
Since $X_1,Y_1\in \mathcal{X}(M_1)$, we have $M_1(X_1)=X_1/\pi X_1\simeq Y_1/\pi Y_1=M_1(Y_1)$ as vector spaces. Therefore it now suffices to show that, for each $j\in\{1,\ldots,n-1\}$, we have $M_j(X_1)/M_{j+1}(X_1)\simeq M_j(Y_1)/M_{j+1}(Y_1)$. To this end, we let $j\in\{1,\ldots,n-1\}$ and note that $M_j\supseteq M_{j+1}\supseteq\pi M_1$. Then we compute that 
$$M_j(X_1)/M_{j+1}(X_1)\cong (X_1\cap M_j)/(X_1\cap M_{j+1})=(X_1\cap M_j)/((X_1\cap M_j)\cap M_{j+1})\cong ((X_1\cap M_j)+M_{j+1})/M_{j+1}.$$
Similarly, we have
$$M_j(Y_1)/M_{j+1}(Y_1)\cong ((Y_1\cap M_j)+M_{j+1})/M_{j+1}.$$
Then we compute that
$$(Y_1\cap M_j)+M_{j+1}=((X_1+\pi M_1)\cap M_j)+M_{j+1}=((X_1\cap M_j)+\pi M_1)+M_{j+1}=(X_1\cap M_j)+M_{j+1}.$$
Hence $M_j(X_1)/M_{j+1}(X_1)\cong M_j(Y_1)/M_{j+1}(Y_1)$, as desired.
\end{proof}

To make use of the above proposition, we decompose the sum on the right-hand side of (\ref{first_application_filtered}). Indeed, by letting $Y_1$ range over submodules of $M_1$ that contain $\pi M_1$, we have
\begin{align*}
    uZ(M;z,w)&=\sum_{Y_1}\sum_{\substack{X_1\\ X_1+\pi M_1=Y_1}}v^{\ell(M_1/X_1)}P(M_* X_1;t)\\
    &=\sum_{Y_1}\sum_{\substack{X_1\\ X_1+\pi M_1=Y_1}}v^{\ell(M_1/X_1)}P(M_* Y_1;t)\\
    &=\sum_{Y_1}P(M_* Y_1;t)\sum_{\substack{X_1\\ X_1+\pi M_1=Y_1}}v^{\ell(M_1/X_1)}.
\end{align*}
We used Proposition \ref{proposition_vector_space_independence} in the second step to write $P(M_*X_1;t)=P(M_*Y_1;t)$.

\subsection{Reduced Hermite normal form}

Now we want to find, for each $Y_1$, a polynomial $Q(Y_1;v)\in\mathbb{Z}[v]$ satisfying 
$$\sum_{\substack{X_1\\ X_1+\pi M_1=Y_1}}v^{\ell({}_\Delta M_1/X_1)}=Q(Y_1;v)\cdot Z(M_1;v).$$
To this end, fix a $Y_1$. Let $Y'$ be the left ideal of $\Lambda=\M_r(\Delta)$ whose rows lie in $Y_1$. By Morita theory, the above sum is equal to
\begin{align}\label{reduced_hey_zeta}
    \sum_{\substack{X'\\ X'+\pi \Lambda=Y'}}v^{\ell({}_{\Lambda}\Lambda/X')},
\end{align}
where $X'$ ranges over finite index left ideals of $\Lambda$. Let $D$ be the classical division ring of fractions of $\Delta$, and let $A=\M_r(D)$. Recall that we may use Hermite normal forms to identify the orbit space $\Lambda^\times\backslash(\Lambda\cap A^\times)$ with a complete set of representatives in $\Lambda\cap A^\times$. Indeed, these representatives take the form
\begin{align}\label{hermite_normal_form}
    \begin{pmatrix}
    \pi^{\ell_1} & a_{1,2} & a_{1,3} & \cdots & a_{1,r}\\
    0 & \pi^{\ell_2} & a_{2,3} & \cdots & a_{2,r}\\
    0 & 0 & \pi^{\ell_3} & \cdots & a_{3,r}\\
    \vdots & \vdots & \vdots & \ddots & \vdots\\
    0 & 0 & 0 & \cdots & \pi^{\ell_r}
\end{pmatrix},
\end{align}
where each $\ell_i$ is a nonnegative integer and each $a_{i,j}$ lies in a complete set of representatives of $\Delta/(\pi^{\ell_j})$. Under this identification, $x\mapsto \Lambda x$ defines a bijection $\Lambda^\times \backslash(\Lambda\cap A^\times)\to\mathcal{X}({}_{\Lambda}\Lambda)$. So our goal is now to understand how the condition $X'+\pi \Lambda=Y'$ pulls back through this bijection.

Let $\bar{\Lambda}=\M_r(k)$. Note that $Y'$ determines a unique orbit in $\bar{\Lambda}^\times\backslash \bar{\Lambda}$ represented by $\bar{y}\in \bar{\Lambda}$, say. Consider the canonical ring homomorphism $\Lambda\twoheadrightarrow \bar{\Lambda}$ given by entrywise reduction mod $\pi\Delta$. Take any lift $y\in \Lambda$ of $\bar{y}\in \bar{\Lambda}$. On groups of units, the map $\Lambda^\times\twoheadrightarrow \bar{\Lambda}^\times$ is still surjective. Therefore $\bar{\Lambda}^\times\backslash \bar{\Lambda}=\Lambda^\times\backslash\bar{\Lambda}$. Now observe that we have the composition of canonical left $\Lambda^\times$-equivariant maps
$$\Lambda\cap A^\times\hookrightarrow \Lambda\twoheadrightarrow \bar{\Lambda}.$$
The orbit $\Lambda^\times\bar{y}\subseteq \bar{\Lambda}$ first pulls back to $\Lambda^\times y+\pi \Lambda=\Lambda^\times(y+\pi \Lambda)\subseteq \Lambda$, which then pulls back to $(\Lambda^\times y+\pi\Lambda)\cap A^\times\subseteq \Lambda\cap A^\times$. Write
$$\mathcal{O}(x)=(\Lambda^\times x+\pi\Lambda)\cap A^\times$$
for any $x\in\Lambda$. Hence sum (\ref{reduced_hey_zeta}) becomes
\begin{align}\label{orbit_hey}
    \sum_{x\in \Lambda^\times\backslash \mathcal{O}(y)}v^{\ell({}_{\Lambda}\Lambda/\Lambda x)},
\end{align}
where we identify $\Lambda^\times\backslash \mathcal{O}(y)$ with a complete set of representatives in $\mathcal{O}(y)$. 

\begin{proposition}\label{proposition_diagonalisation}
For any $a,b\in\Lambda^\times$ and $c\in\pi\Lambda$, we have
\begin{align*}
    \sum_{x\in \Lambda^\times\backslash \mathcal{O}(y)}v^{\ell({}_{\Lambda}\Lambda/\Lambda x)}=\sum_{x\in \Lambda^\times\backslash \mathcal{O}(ayb+c)}v^{\ell({}_{\Lambda}\Lambda/\Lambda x)}.
\end{align*}
\end{proposition}
\begin{proof}
    From the formula $\mathcal{O}(y)=\left(\Lambda^\times y+\pi\Lambda\right)\cap A^\times$, we see that $\mathcal{O}(y)=\mathcal{O}(ay)$ and $\mathcal{O}(y)=\mathcal{O}(y+c)$ for any $y\in\Lambda$. So it suffices to check the proposition with $a=1$, $c=0$ and $b$ arbitrary. To see the claim in this more limited case, first note that right multiplication by $b$ defines a $\Lambda^\times$-equivariant automorphism of $\Lambda$ that takes $\mathcal{O}(y)$ to $\mathcal{O}(yb)$ for any $y\in\Lambda$. Then we check that $\Lambda/(\Lambda xb)\cong (\Lambda b^{-1})/(\Lambda x)=\Lambda/\Lambda x$, as left $\Lambda$-modules, for any $x\in\Lambda$.
\end{proof}

By Proposition \ref{proposition_diagonalisation}, in dealing with sum (\ref{orbit_hey}), we may assume that $y=\diag(1,\ldots,1,0,\ldots,0)\in \Lambda$ with $m\le r$ ones, say. Therefore the Hermite normal forms (\ref{hermite_normal_form}) representing the orbits in $\Lambda^\times\backslash \mathcal{O}(y)$ are precisely those with $\ell_1=\cdots=\ell_m=0$ and all other entries in $\pi\Delta$. Then we compute that 
\begin{align*}
    \sum_{x\in \Lambda^\times\backslash \mathcal{O}(y)}v^{\ell({}_{\Lambda}\Lambda/\Lambda x)}&=\prod_{j=m+1}^r \sum_{\ell_j=1}^\infty q^{(j-1)(\ell_j-1)}v^{\ell_j}\\
    &=\prod_{j=m+1}^r \frac{v}{1-q^{j-1}v}\\
    &= v^{r-m}\prod_{i=1}^m (1-q^{i-1}v)\cdot \prod_{j=1}^r \frac{1}{1-q^{j-1}v}.
\end{align*}
So we take
$$Q(Y_1;v)=v^{r-m}\prod_{i=1}^m (1-q^{i-1}v),$$
noting that $m=\dim_k(Y_1/\pi M_1)$.

\subsection{Bushnell-Reiner-Solomon theorem}\label{BRS_factor}

Bushnell and Reiner \cite{Bushnell1980} refined and proved Solomon's first conjecture \cite{Sol_79_second}. Later, they gave an even more refined version, for Bushnell-Reiner zeta functions \cite{Bushnell1987}. We give the following version over hereditary orders.

\begin{theorem}
    Let $R$ be a basic indecomposable one-dimensional semiperfect regular ring, with Jacobson radical $J$, such that $R/J$ is finite. Then we have an Artin-Wedderburn decomposition $R/J\simeq \prod_i \mathbb{F}_q$. Let $M$ be a finitely generated projective left $R$-module. Then $M$ is the direct sum of $r$ principal indecomposables, say. Recall the refined Bushnell-Reiner zeta function (\ref{refined_brz})
    $$Z(M;z,w)=\sum_{X}z^{\ell(M/X)}w^{\ell(X/JX)}\in\mathbb{Z}\llbracket z\rrbracket[w],$$
    where $z=(z_i)$ and $w=(w_i)$ are both $\mathcal{S}(R)$-tuples of variables. Let $v=\prod_i z_i$. Then there is a polynomial $F(M;z,w)\in\mathbb{Z}[z,w]$ such that
    $$Z(M;z,w)=F(M;z,w)\cdot\prod_{j=0}^{r-1}(1-q^j v)^{-1}.$$
\end{theorem}
\begin{proof}
By Michler's classification theorem \cite{Michler1969} (also \cite[Proposition~1.6]{rump_reg}), we may identify
$$R=\begin{pmatrix}\Delta & \Delta & \Delta & \cdots & \Delta\\
(\pi) & \Delta & \Delta & \cdots & \Delta\\
(\pi) & (\pi)& \Delta & \cdots & \Delta\\
\vdots & \vdots & \vdots & \ddots & \vdots\\
(\pi) & (\pi)& (\pi) & \cdots & \Delta
\end{pmatrix}\subseteq\M_{n\times n}(\Delta),$$
where $\Delta$ is a noncommutative discrete valuation ring with uniformiser $\pi$ and finite residue field $\Delta/(\pi)\simeq\mathbb{F}_q$. Here, $n=|\mathcal{S}(R)|$. Then we have the canonical ring isomorphism
$$R/J\cong\begin{pmatrix}\Delta/(\pi) & 0 & 0 & \cdots & 0\\
0 & \Delta/(\pi) & 0& \cdots & 0\\
0 & 0& \Delta/(\pi) & \cdots & 0\\
\vdots & \vdots & \vdots & \ddots & \vdots\\
0& 0& 0 & \cdots & \Delta/(\pi)
\end{pmatrix}\subseteq\M_{n\times n}(\Delta/(\pi)).$$
This gives a natural ordering to our variables: $z_1,\ldots,z_n$ and $w_1,\ldots,w_n$ (corresponding to the entry along the main diagonal). Moreover, we may identify
$$M=(C_1 | \cdots | C_r)\subseteq\M_{n\times r}(\Delta),$$
where $C_1,\ldots,C_r$ are columns of $R$, allowing for any number of repetitions but keeping their natural ordering. Let $M_1,\ldots,M_n$ be the rows of $M$, and put $M_{n+1}=\pi M_1$. For each submodule $Y_1$ of ${}_\Delta M_1$ containing $\pi M_1$, consider the filtered vector space 
$$M_*Y_1:M_1(Y_1)\supseteq \cdots \supseteq M_n(Y_1)\supseteq M_{n+1}(Y_1)=0,$$
where 
$$M_j(Y_1)=(Y_1\cap M_j)/\pi Y_1$$
for each $j\in\{1,\ldots,n\}$. Also for each such $j$, put
$$t_j=w_j\prod_{i=j+1}^n z_i.$$
For each submodule $Y_1$ of ${}_\Delta M_1$ containing $\pi M_1$, we consider the following two polynomials. Firstly, the polynomial $P(M_*Y_1;t)$, as defined in (\ref{appendix_polynomial_filtered_vector_spaces}). Secondly, the polynomial $$Q(Y_1;v)=v^{r-m}\prod_{i=1}^m (1-q^{i-1}v),$$
where $m=\dim_{\Delta/(\pi)}(Y_1/\pi M_1)$. Now let 
$$F(M;z,w)=\frac1{u}\sum_{Y_1} P(M_* Y_1;t)\cdot Q(Y_1;v),$$
where $u=\prod_{i=1}^n z_i^{\ell(M_1/M_i)}$. By construction, we have 
\begin{align*}
    Z(M;z,w)=Z({}_\Delta M_1;v)\cdot F(M;z,w).
\end{align*}
Comparing coefficients in this equation, we see that $F(M;z,w)\in\mathbb{Z}[z,w]$. Finally, we have $M_1/\pi M_1\simeq \mathbb{F}_q^{\oplus r}$. Therefore,
\begin{align*}
    Z({}_\Delta M_1;v)=\prod_{j=0}^{r-1}(1-q^j v)^{-1}
\end{align*}
by our abstract Hey formula (Appendix \ref{section_hey}).  
\end{proof}

\begin{remark}
The classical interpretation of Solomon's first conjecture is somewhat different. The following observation is more in keeping with the classical interpretation. Recall that $R$ is an order in a simple Artinian ring $A$. Then we see that $A\otimes_R M$ has length $r$ as a left $A$-module. Hence, the factor $\prod_{j=0}^{r-1}(1-q^j v)^{-1}$ remains unchanged if we replace $M$ with another finitely generated projective left $R$-module $M'$ such that $\length_A(A\otimes_R M')=r$.

    Bushnell and Reiner proved Solomon's first conjecture in \cite[\S.~4.1, pp.~147--150]{Bushnell1980}. Their proof reduces the requisite comparison problem to computing a particular $p$-adic integral, in the second display equation on \cite[p.~150]{Bushnell1980}. When written as a sum over orbits, their integral is essentially our sum in (\ref{orbit_hey}). After a diagonalisation argument (like Bushnell and Reiner used earlier in their proof), we end up with the same sum as Bushnell and Reiner.
\end{remark}

\printbibliography

\bigskip
\noindent\begin{minipage}{\textwidth}
\small
Sean B. Lynch, Institute for Theoretical Sciences, Westlake University in Hangzhou, Zhejiang, China\\
\textit{E-mail address:} \href{mailto:sblynch@westlake.edu.cn}{\texttt{sblynch@westlake.edu.cn}}
\end{minipage}

\end{document}